\documentclass[10pt,a4paper,leqno]{article}
\usepackage[utf8]{inputenc}
\usepackage[english]{babel}
\usepackage{amsmath,amsthm,amsfonts,amssymb}
\usepackage{dsfont}
\usepackage{enumerate}

\newcommand{\R}{\mathbb{R}} 
\newcommand{\C}{\mathbb{C}} 
\newcommand{\N}{\mathbb{N}} 
\newcommand{\Z}{\mathbb{Z}} 

\newcommand{\RN}{\mathbb{R}^{N}}
\newcommand{\crit}{{2^{\ast}}}
\newcommand{\crid}{{2^{+}}}
\newcommand{\sqe}{\sqrt{\varepsilon}}
\newcommand{\eps}{\varepsilon}
\newcommand{\dx}{\,dx}
\newcommand{\dy}{\,dy}
\newcommand{\wh}{\widehat}
\newcommand{\wt}{\widetilde}
\newcommand{\weakto}{\rightharpoonup}
\newcommand{\intrn}{\int\limits_{\RN}}
\newcommand{\cR}{\mathcal{R}}
\newcommand{\mR}{\mathbf{R}}

\newcommand{\ce}{\mathcal{C}} 
\newcommand{\ceinfty}{\mathcal{C}^{\infty}} 
\newcommand{\ceinftyc}{\mathcal{C}^{\infty}_{c}} 
\newcommand{\SR}{\mathcal{S}} 
\newcommand{\cS}{\mathcal{S}} 
\newcommand{\SRd}{\mathcal{S}'} 

\newtheorem{theorem}{\\Theorem}[section] 
\newtheorem{remark}[theorem]{\\Remark}
\newtheorem{lemma}[theorem]{\\Lemma}
\newtheorem{proposition}[theorem]{\\Proposition}

\begin{document}

\title{Dual ground state solutions for the critical nonlinear Helmholtz equation}
\author{Gilles Ev\'equoz and Tolga Ye\c{s}il\\  \\ 
{\small Institut f\"ur Mathematik} \\ {\small Goethe-Universit\"at Frankfurt am Main} \\ {\small Germany}}
 \date{}
\maketitle

\begin{abstract}
Using a dual variational approach we obtain nontrivial real-valued solutions of the critical nonlinear Helmholtz equation
$$ - \Delta u - k^{2}u = Q(x)|u|^{\crit - 2}u, \quad u \in W^{2,\crit}(\RN) $$
for $N\geq 4$, where $\crit := \frac{2N}{N-2}$. The coefficient $Q \in L^{\infty}(\RN)\setminus\{0\}$ is assumed to be nonnegative, 
asymptotically periodic and to satisfy a flatness condition at one of its maximum points. The solutions obtained are
so-called {\em dual ground states}, {\it i.e.}, solutions arising from critical points of the dual functional 
with the property of having minimal energy among all nontrivial critical points. Moreover, we show that no dual ground state exists for $N=3$.
\end{abstract}

\section{Introduction}
In this paper, we focus our attention on the existence of nontrivial real-valued solutions of the critical nonlinear
Helmholtz equation
\begin{equation}\label{eqn:critical_helmholtz}
-\Delta u - k^2 u = Q(x)|u|^{p - 2}u, \quad u \in W^{2,p}(\RN)
\end{equation}
for $N \geq 3$, $k\neq 0$, and where $Q \in L^{\infty}(\RN)\setminus\{0\}$ is a nonnegative weight function 
and $p = \crit := \frac{2N}{N-2}$ is the critical Sobolev exponent.
Recently \cite{Evequoz2015}, the existence of solutions to \eqref{eqn:critical_helmholtz} has been proven for all $p$ in the noncritical
interval $\left(\frac{2(N+1)}{N-1}, \frac{2N}{N-2}\right)$. 
Due to the lack of a direct variational approach, since the associated 
energy functional is not well defined in $W^{2,p}(\RN)$, the authors 
considered instead the integral equation
\begin{equation*}\label{eqn:real_integral_equation}
u = \textbf{R}_{k}\left(Q |u|^{p-2}u\right), \quad u \in L^{p}(\RN),
\end{equation*}
where $\textbf{R}_k $ denotes the real part of the resolvent operator of $-\Delta -k^2$. 
A dual variational approach was used,
based on the dual energy functional $J_{Q,p}$ given by
\begin{equation*}\label{dual_functional}
J_{Q,p} (v) = \frac{1}{p'}\intrn |v|^{p'} \,dx - \frac{1}{2}\intrn v \mathbf{A}_{Q,p}v \,dx,\quad v\in L^{p'}(\RN),
\end{equation*}
where $p'=\frac{p}{p-1}$ and $\mathbf{A}_{Q,p} v = Q^{\frac{1}{p}}\mathbf{R}_{k}\left(Q^{\frac{1}{p}}v\right)$.
The functional $J_{Q,p}$ is of class $C^1$ and has the mountain pass geometry. 
Therefore the properties of $Q$ determine whether it satisfies the Palais-Smale condition 
and this in turn is linked in an essential way to compactness properties of 
the Birman-Schwinger type operator $\mathbf{A}_{Q,p}$.
For noncritical $p$, the operator $\mathbf{A}_{Q,p}$ is compact if $Q$ vanishes at infinity, and $J_{Q,p}$
satisfies the Palais-Smale condition. When $Q$ is periodic, this is not the case anymore,
but $\textbf{A}_{Q,p}$ still has some local compactness. In combination with a crucial 
nonvanishing property \cite[Theorem 3.1]{Evequoz2015} of the quadratic form associated with $\mathbf{R}_k$, 
a nontrivial critical point can then be obtained as weak limit after translation of a Palais-Smale sequence 
at the mountain pass level.
The problem \eqref{eqn:critical_helmholtz} becomes more delicate in the critical case $p = \crit$. 
Applying to the differential equation \eqref{eqn:critical_helmholtz} the rescalings
\begin{equation}\label{eqn:rescaling}
u \mapsto u_{r,x_{0}}, \quad \text{where } \quad u_{r,x_{0}}(x) = r^{\frac{N-2}{2}}u(r(x-x_{0})),
\end{equation}
the linear term vanishes as $r \to \infty$ and, since the limit problem
\begin{equation}\label{eqn:limit_pb}
- \Delta u  = Q(x_{0})|u|^{\crit - 2}u \quad \text{in } \RN 
\end{equation}
possesses nontrivial solutions, the local compactness of $\mathbf{A}_{Q}:=\mathbf{A}_{Q,\crit}$ is lost.
In the case where $Q$ vanishes at infinity, the functional $J_{Q}:=J_{Q,\crit}$ therefore does not
satisfy the Palais-Smale condition at every level. 

In analogy to the study of the critical problem \eqref{eqn:critical_helmholtz} on a bounded domain, 
starting with the celebrated work of Br\'ezis and Nirenberg~\cite{Brezis-Nirenberg83},
we shall try to recover some kind of compactness by comparing the mountain pass level $L_Q$ of the 
functional $J_Q$ with the least energy level $L_Q^\ast$ among all possible
limiting problems \eqref{eqn:limit_pb} with $x_0\in\R^N$. From the duality between the Sobolev and 
the Hardy-Littlewood-Sobolev inequalities, it follows that
$$
L_Q^\ast=\frac{S^{\frac{N}2}}{N\|Q\|_\infty^{\frac{N-2}2}},
$$ 
where $S$ denotes the optimal constant in the Sobolev inequality (see Section 3.2 for more details).

The general strategy consists roughly in two steps:
\begin{itemize} 
\item[(I)] show that at every level $0<\beta<L_Q^\ast$, the Palais-Smale condition is satisfied, and 
\item[(II)] establish the strict inequality $L_Q<L_Q^\ast$. 
\end{itemize}
Ambrosetti and Struwe~\cite{Ambrosetti-Struwe86} confirmed that, for the Dirichlet problem on a bounded domain, 
this scheme is also adapted to the dual variational framework. However, whereas the authors in \cite{Ambrosetti-Struwe86} 
reduce the proof of the Palais-Smale condition for the dual functional to the proof of the same property for the direct functional,
we do not have for the problem \eqref{eqn:critical_helmholtz} on $\R^N$ a direct functional at hand. 
In our approach towards the above steps (I) and (II), we choose instead to work directly with the resolvent operators 
for the original and the limit problems, via the corresponding fundamental solutions. More precisely, we start by 
deriving accurate upper and lower bounds on the difference of these fundamental solutions 
(Lemma~\ref{lem:difference_fundsol}). This involves a detailed study of Bessel functions for small arguments. 
Based on these estimates, we then prove a new local compactness property for the difference operator 
$\mathbf{R}_k-\mathbf{R}_0$, where $\mathbf{R}_0=(-\Delta)^{-1}$ (see Proposition~\ref{prop:lokkomp}).
In addition, we show that $\mathbf{R}_k$ remains locally compact in subcritical Lebesgue spaces. Combining these 
properties, the step (I) can be completed in the case where $Q$ vanishes at infinity.

The next step is to prove the strict inequality $L_Q<L_Q^\ast$. There, the lower bound on the difference of the fundamental solutions 
plays a key role. Indeed, it implies that in dimension $N\geq 4$ the quadratic form of the operator $\mathbf{R}_k-\mathbf{R}_0$ is positive
for positive functions supported in sets of small diameter. For such functions, the energy of $J_Q$ can thus be made smaller 
than that of the dual functional associated to \eqref{eqn:limit_pb}. Since we are working with a nonconstant $Q$, an additional requirement 
(see (Q2) in Theorem \ref{thm:main} below) is needed to complete the argument. The condition that we impose controls the way in which $Q$ 
approaches its maximum value $\|Q\|_\infty$. The same condition also appears in several related critical problems, and it seems to go back 
to the work of Escobar~\cite{Escobar87}. Let us mention that Egnell~\cite{Egnell88} provided examples of critical problems on bounded domains 
for which this assumption is necessary. 
More recently, this condition was also used in a paper by Chabrowski and Szulkin~\cite{Chabrowski-Szulkin2002}, 
on a strongly indefinite critical nonlinear Schrödinger equation on $\R^N$ with periodic coefficients.
There, the authors work in a direct variational framework and use generalized linking arguments to show the existence 
of a Palais-Smale sequence at some level. The condition (Q2) is used to prove that this level lies strictly below $L_Q^\ast$. 
A nontrivial critical point is then obtained with the help of Lions' local compactness Lemma~\cite{Lions84} 
(see also \cite[Lemma 1.21]{Willem}). Our approach to treat periodic, 
and more generally asymptotically periodic functions $Q$, is inspired by \cite{Chabrowski-Szulkin2002}, 
although our arguments differ significantly. Working within the dual framework, 
we can simply use the mountain pass theorem without Palais-Smale condition, but we need to show that the nonvanishing property 
for $\mathbf{R}_k$, proven in \cite[Theorem 3.1]{Evequoz2015} for noncritical exponents, continues to hold in the critical case $p=\crit$.

As already pointed out by Br\'ezis and Nirenberg~\cite{Brezis-Nirenberg83}, there is a strong contrast between 
the dimensions $N=3$ and $N\geq 4$, for problems with the critical exponent. In the present case, the estimates on the difference of
the fundamental solutions have the opposite sign for $N=3$, so that $\mathbf{R}_k-\mathbf{R}_0$ acts negatively 
on positive functions. This does not permit to verify Step (II) above, and we show that in fact $L_Q=L_Q^\ast$ holds
for any bounded $Q\geq 0$ in this case. Moreover, we find that the mountain pass level $L_Q$ is not achieved.

As indicated in \cite{Evequoz2015}, every nontrivial critical point $v\in L^{\frac{2N}{N+2}}(\RN)$ of $J_Q$ is related, via the transformation
$$
u=\mathbf{R}_k\bigl(Q^\frac1{\crit}v\bigr),
$$
to a nontrivial strong solution $u\in W^{2,\crit}(\RN)$ of \eqref{eqn:critical_helmholtz} (see Section 3.1 for more details). 
The solutions we obtain in the present paper have the distinctive property that the corresponding critical point of $J_Q$ 
has minimal energy among all nontrivial critical points. 
Following the terminology introduced in a recent paper \cite{Evequoz_ampa2017}, 
we call such solutions {\em dual ground states} of \eqref{eqn:critical_helmholtz} (cf. Section 3.1 for the precise definition). 
The main result in the present paper is the following.
\begin{theorem}\label{thm:main}
Let $N\geq 3$ and consider $Q\in L^\infty(\R^N)\backslash\{0\}$ such that $Q\geq 0$ a.e. in $\R^N$.
\begin{itemize} 
\item[(i)] If $N\geq 4$ and $Q$ satisfies the following conditions,
\begin{itemize}
\item[(Q1)] $Q=Q_{\text{per}}+Q_0$, where $Q_{\text{per}}, Q_0\geq 0$ are such that $Q_{\text{per}}$ is periodic 
and $Q_0(x)\to 0$ as $|x|\to\infty$;
\item[(Q2)] there exists $x_{0} \in \RN$ with $Q(x_{0}) = \max\limits_{\RN}Q$ and, as $|x-x_0|\to 0$,
$$
Q(x_0) - Q(x) = \begin{cases} o(|x - x_{0}|^2),\quad \text{if }N\geq 5, \\ \\ O(|x-x_0|^2), \quad\text{if }N=4,\end{cases} 
$$
\end{itemize}
then, the problem \eqref{eqn:critical_helmholtz} with $p=2^\ast$ has a dual ground state.

\item[(ii)] If $N=3$, no dual ground state exists for \eqref{eqn:critical_helmholtz} with $p=2^\ast$.
\end{itemize}
\end{theorem}
Note that in the first part, we do not exclude the cases $Q=Q_0$ and $Q=Q_{\text{per}}$.

\medskip

Let us point out that for radially symmetric $Q$, radial solutions of \eqref{eqn:critical_helmholtz} 
for all $p\in(2,\infty)$ have been obtained in \cite{Evequoz-Weth2014} and very recently in \cite{Mandel2017}, 
where a broad class of nonlinearities is considered.
Up to our knowledge, Theorem~\ref{thm:main} is the first result concerning solutions of 
the nonlinear Helmholtz equation with critical nonlinearity and nonradial $Q$.
Let us also mention that the lower critical case $p=\frac{2(N+1)}{N-1}$ is still open. There, we expect completely
different phenomena than for $p=\crit$. A suitable method therefore needs to be found and we will
address this issue in a forthcoming paper.

We shall now briefly describe the structure of the paper. 
In Section 2, we study the Helmholtz resolvent operator in the Lebesgue space $L^{\crid}(\RN)$, 
where $\crid=\frac{2N}{N+2}$ is the conjugate exponent to $\crit$. 
Recalling first the construction of the fundamental solution of the Helmholtz equation and its asymptotic properties,
we derive in Lemma~\ref{lem:difference_fundsol} new upper and lower bounds on the difference of the latter and the
fundamental solution of Laplace's equation, for small arguments. 
The proof consists in the precise estimation of Bessel functions and their derivatives, 
and the result is of crucial importance for the whole paper.
As a first application, we prove in Proposition~\ref{prop:lokkomp} that the difference 
$\textbf{R}_k-\textbf{R}_0$: $L^{\crid}(\RN)$ $\to$ $L^{\crit}_{\text{loc}}(\RN)$, where $\textbf{R}_{0}$ 
denotes the Laplace resolvent operator, is compact. 
There, we start by decomposing the fundamental solution of the Helmholtz equation
in a similar way as in \cite{Evequoz2015} and then apply the upper bounds obtained in Lemma~\ref{lem:difference_fundsol}.
Another essential property of the Helmholtz resolvent, the nonvanishing property, is established in the case $p = \crit$ in
Theorem~\ref{nonvanishing}. Its proof relies on improvements of previous results from \cite{Evequoz2015} 
by means of the Hardy-Littlewood-Sobolev inequality. After this study of the Helmholtz resolvent, we turn in Section 3 
to the existence of dual ground states of \eqref{eqn:critical_helmholtz} with $p=2^\ast$.
We start by recalling the dual variational framework set up in \cite{Evequoz2015} and the characterization 
of the dual mountain pass level $L_Q$. Using the compactness properties of $\mathbf{R}_k$ and of $\mathbf{R}_k-\mathbf{R}_0$ 
established in Section 2, we then analyze the behaviour of Palais-Smale sequences for $J_Q$ at the level $L_Q$.
Under the assumption $L_Q<L_Q^\ast$, we obtain in Proposition~\ref{prop:PS_unter_kritisch_Level} 
the existence of a nontrivial critical point for $J_Q$ in the case where $Q$ is asymptotically periodic. 
The nonvanishing property plays here a key role in handling the periodic part $Q_{\text{per}}$ 
of the coefficient $Q$. Section 3.3 is then devoted to estimating the dual mountain pass level $L_Q$
under the additional ``flatness condition'' ($Q2$). There, we show that the positive lower bound on the difference of
the fundamental solutions given by Lemma~\ref{lem:difference_fundsol} yields the strict inequality $L_Q<L_Q^\ast$, in the case $N\geq 4$.
Combining the above results, we obtain in Section 3.4 the existence of dual ground states stated in Theorem~\ref{thm:main}.
The paper concludes with the $3$-dimensional case, in which we show that $L_Q=L_Q^\ast$ holds and, by a contradiction argument,
we obtain the nonexistence of dual ground states for \eqref{eqn:critical_helmholtz} with $p=2^\ast$ in this case.

\bigskip

We close this introduction by fixing some notation. 
Throughout the paper we denote by $B_{r}(x)$ the open ball in $\RN$ with radius $r$ and center at $x$. 
Moreover, we set $B_{r} = B_{r}(0)$. The constant $\omega_{N}:=\frac{2\pi^{\frac{N}2}}{N\Gamma(\frac{N}2)}$, 
where $\Gamma$ is the gamma function, represents the volume of the unit ball $B_1$.
By $\mathds{1}_{M}$ we shall denote the characteristic function of 
a measurable set $M \subset \RN$. We write $\SR(\RN)$ for the space of Schwartz functions and 
$\SRd$ for its dual, i.e., the space of tempered distributions. 
Furthermore, we let $\widehat{f}$ or ${\mathcal{F}}(f)$ denote the Fourier transform 
of a function $f \in \SRd$. 
For $1\leq s\leq \infty$, we abbreviate the norm in $L^{s}(\RN)$ by $\left\|\cdot\right\|_{s}$.

\section{The Helmholtz resolvent in the critical case}
\subsection{Fundamental solutions}
Without loss of generality and to simplify formulas, we consider the problem \eqref{eqn:critical_helmholtz} with $k=1$. 
The general case follows by rescaling the independent variable.
 
For $N \geq 3$, the radial outgoing fundamental solution of the Helmholtz equation $-\Delta u-u=\delta_0$ in $\RN$ is given by
\begin{equation}\label{fundsol}
\Phi(x):= \frac{i}{4}(2\pi|x|)^{\frac{2-N}{2}}H^{(1)}_{\frac{N-2}{2}}(|x|), \quad\text{for }x\in\RN\backslash\{0\},
\end{equation}
where $H^{(1)}_{\frac{N-2}{2}}$ denotes the Hankel function of the first kind of order $\frac{N-2}{2}$.
For $f \in \cS(\RN)$ the function $u := \Phi\ast f \in \mathcal{C}^{\infty}(\RN)$ is a solution 
of the inhomogeneous Helmholtz equation $-\Delta u - u = f$ which satisfies the Sommerfeld 
outgoing radiation condition $\partial_ru(x)-iu(x)=o(|x|^{\frac{1-N}{2}})$, as $|x|\to\infty$.
Moreover, it is known (see \cite{Gelfand1964}) that, in the sense of tempered distributions, the Fourier transform of $\Phi$
is given by
\begin{equation}\label{eqn:Fourier_Phi}
\wh{\Phi}(\xi)=(2\pi)^{-\frac{N}2}\frac{1}{|\xi|^2-(1+i0)}:=(2\pi)^{-\frac{N}2}\lim_{\eps\to 0^+}\frac{1}{|\xi|^2-(1+i\eps)}.
\end{equation}

Since we shall be considering real-valued solutions of the Helmholtz equation in the sequel, we turn our attention to 
\begin{equation}\label{eqn:realteil}
\Psi(x):=\text{Re}[\Phi(x)] = -\frac14\bigl(2\pi |x|\bigr)^{\frac{2-N}{2}} Y_{\frac{N-2}{2}}(|x|), \quad\text{for }x\in\RN\backslash\{0\},
\end{equation}
where $Y_{\frac{N-2}{2}}$ denotes the Bessel function of the second kind of order $\frac{N-2}{2}$.
$\Psi$ should be seen as the fundamental solution of the Helmholtz equation associated to real-valued standing waves.

Let us recall some well-known facts concerning Bessel functions of the second kind:
For nonnegative orders $\nu$ and positive 
arguments $t$, the asymptotic behaviour of $Y_\nu(t)$ is given by 
(see \cite[Remark 5.16.2]{Lebedev2012})
\begin{align}
Y_{\nu}(t) &= - \frac{2^{\nu}\Gamma(\nu)}{\pi t^{\nu}}\Bigl(1+O(t)\Bigr), && \text{as } t \to 0, \quad\text{if }\nu>0,\label{gegennull}\\
Y_{0}(t) &= - \frac{2}{\pi}\ln{\frac{2}{t}}+O(1), && \text{as } t \to 0,\label{orderzero}\\
Y_{\nu}(t) &= -\sqrt{\frac{2}{\pi t}}\cos\bigl(t - \frac{(2\nu-1)\pi}{4}\bigr)\Bigl(1+O(t^{-1})\Bigr), && 
\text{as } t \to \infty,\quad\text{for all }\nu\geq 0. \label{unendlich} 
\end{align}
As a consequence, we find that
\begin{equation}\label{asympt:Psi}
\Psi(x)=\begin{cases} \frac1{N(N-2)\omega_N} |x|^{2-N}\bigl(1+O(|x|)\bigr), & \text{ as }|x|\to 0,\\ \\ 
 \frac12(2\pi|x|)^{\frac{1-N}2}\cos\bigl(|x|-\frac{(N-3)\pi}{4}\bigr) \bigl(1+O(|x|^{-1})\bigr), &\text{ as }|x|\to\infty.\end{cases}
\end{equation} 
Denoting by $y_{\nu}$ the first positive zero of $Y_{\nu}$ with $\nu \geq 0$, 
we deduce from the asymptotics \eqref{gegennull} and  \eqref{orderzero}, 
that $Y_{\nu}(t) < 0$ for all $t \in (0,y_{\nu})$ and therefore $\Psi(x)>0$ for all $|x|<y_{\nu}$.

Recalling that for $N\geq 3$ the fundamental solution $\Lambda$ of Laplace's equation in $\R^N$ is given by
\begin{equation}\label{eqn:laplacefundsol}
\Lambda(x)=\frac{1}{N(N-2)\omega_N}|x|^{2-N},\quad\text{for }x\in\RN\backslash\{0\},
\end{equation}
we see from \eqref{asympt:Psi} that $\Psi(x)$ behaves like $\Lambda(x)$ for small $|x|$. Our first result
gives more precise estimates on the way $\Psi(x)$ approaches $\Lambda(x)$ as $|x|\to 0$. In particular,
we observe a strong contrast between the dimension $N=3$ and the higher dimensions $N\geq 4$.
\begin{lemma}\label{lem:difference_fundsol}
Let $r>0$ be given such that $r<y_{\frac{N-4}{2}}$ if $N\geq 4$ and $r<\pi$ if $N=3$.
\begin{enumerate}[(i)]
	
\item There exist $\kappa_{1},\kappa_{2}>0$ only depending on $r$ and $N$,  such that for all $x\in B_r$,
$$
 \left\{\begin{array}{lll}
\kappa_{1}|x|^{4-N} &\leq \Psi(x)-\Lambda(x) \leq \kappa_{2} |x|^{4-N}, & \text{if }N\geq 5,\\ \\
\kappa_{1} \bigl|\ln|x|\, \bigr| &\leq \Psi(x)-\Lambda(x) \leq \kappa_{2} \bigl|\ln|x|\, \bigr|, & \text{if }N=4,\\ \\
-\kappa_{1}|x| &\geq \Psi(x)-\Lambda(x) \geq -\kappa_{2} |x|, & \text{if }N=3.
\end{array}\right.
$$
\item For every multiindex $\alpha \in \N^{N}_{0}$ with $|\alpha| \geq 1$, there exists $\kappa_{3}>0$ 
only depending on $|\alpha|$, $r$ and $N$, such that
\begin{equation}\label{eqn:ableitungfundsol}
\bigl| \partial^{\alpha} \left(\Psi(x) - \Lambda(x)\right) \bigr| 
\leq \kappa_{3} |x|^{4-N-|\alpha|}, \quad \text{for all } x \in B_{r}.
\end{equation}
\end{enumerate}
\end{lemma}
\begin{proof}
We start by considering for $\nu\geq 1$ the function $\eta_\nu$: $[0,\infty)$ $\to$ $\R$ given by
$$
\eta_\nu(t):=\begin{cases}-c_\nu t^\nu Y_\nu(t), & t>0, \\ 1, & t=0,\end{cases} 
\quad\text{where }c_\nu=\frac{\pi}{2^\nu\Gamma(\nu)}.
$$
Remark that $\eta_\nu$ is continuous, as a consequence of \eqref{gegennull} and since $Y_\nu$ is analytic on $(0,\infty)$. 
In addition, for $t>0$, the recursion formula $\frac{d}{dt}\left[t^{\nu}Y_{\nu}(t)\right] = t^{\nu}Y_{\nu -1}(t)$ 
(see \cite[p.105]{Lebedev2012}) gives
$$
\eta_\nu'(t)=-c_\nu t^\nu Y_{\nu-1}(t).
$$
Hence, $\eta_\nu$ is strictly increasing in the interval $(0,y_{\nu-1})$ and in particular $\eta_\nu>1$ in this interval.
Moreover, using the asymptotic expansions for small arguments \eqref{gegennull} and \eqref{orderzero}, we see that
\begin{align*}
&\lim_{t\to 0^+}\frac{\eta_\nu'(t)}{t}=-c_\nu \lim_{t\to 0^+} t^{\nu-1}Y_{\nu-1}(t)=\frac1{2(\nu-1)}, \quad\text{if }\nu>1,\\
\text{and}\quad &\lim_{t\to 0^+}\frac{\eta_1'(t)}{t|\ln t|}=-\frac{\pi}{2} \lim_{t\to 0^+}\frac{Y_{0}(t)}{-\ln t}=1.
\end{align*}
Therefore, given $0<r<y_{\nu-1}$ and since $y_0<1$, there exist constants 
$\kappa_1'=\kappa_1'(\nu,r)$ and $\kappa_2'(\nu,r)$ such that
\begin{align*}
\frac{\eta_\nu'(t)}{t}\geq 2\kappa_1', \quad\text{ if }\nu>1, \quad
&\text{ and }\quad \frac{\eta_1'(t)}{t|\ln t|}\geq 2\kappa_1', \quad\text{for all }0<t<r,  \\
\Bigl|\frac{\eta_\nu'(t)}{t}\Bigr|\leq 2\kappa_2', \quad\text{ if }\nu>1, \quad
&\text{ and }\quad \Bigl|\frac{\eta_1'(t)}{t\ln t}\Bigr|\leq 2\kappa_2', \quad\text{for all }0<t<r.
\end{align*}
Writing
\begin{align*}
\frac{\eta_\nu(t)-1}{t^2}=\int_0^1 s\frac{\eta_\nu'(st)}{st}\, ds 
\quad\text{ and }\quad \frac{\eta_1(t)-1}{t^2|\ln t|}=\int_0^1 s\left|\frac{\ln(st)}{\ln t}\right|\frac{\eta_1'(st)}{st|\ln(st)|}\, ds,
\end{align*}
we obtain the bounds
\begin{align}
\eta_\nu(t)-1\geq \kappa_1' t^2,\quad\text{ if }\nu>1, \quad
&\text{ and }\quad \eta_1(t)-1\geq \kappa_1't^2|\ln t|, \quad\text{for all }0<t<r,  \label{eqn:eta_below} \\
\bigl|\eta_\nu(t)-1\bigr|\leq \kappa_2' t^2,\quad\text{ if }\nu>1, \quad
&\text{ and }\quad \bigl|\eta_1(t)-1\bigr|\leq \kappa_2''t^2|\ln t|, \quad\text{for all }0<t<r, \label{eqn:eta_above}
\end{align}
with some $\kappa_2''=\kappa_2'(r)>0$.
The assertion (i) in case $N\geq 4$ follows from \eqref{eqn:eta_below} and \eqref{eqn:eta_above}, since we have
$$
\Psi(x)-\Lambda(x)=\Lambda(x)\bigl(\eta_{\frac{N-2}{2}}(|x|)-1\bigr)
=\frac1{N(N-2)\omega_N}|x|^{2-N}\bigl(\eta_{\frac{N-2}{2}}(|x|)-1\bigr).
$$
In the case $N=3$, we have
\begin{equation}\label{eqn:Ngleich3}
\Psi(x)-\Lambda(x)=\frac1{4\pi|x|}\bigl(\cos|x|-1\bigr).
\end{equation}
Remark that
$$
\frac{\cos t -1}{t^2}=-\int_0^1 s \frac{\sin(st)}{st}\, ds, 
\quad\text{where }t\mapsto \frac{\sin t}{t}\text{ is decreasing in }[0,\pi],\ \lim_{t\to 0}\frac{\sin t}{t}=1,
$$
and $\bigl|\frac{\sin t}{t}\bigr|\leq 1$ for all $t>0$.
We thus conclude that for given $0<r_0<\pi$ there is a constant $\kappa_1=\kappa_1(r_0)>0$ such that
$$
\frac{\cos t -1}{t^2}\leq -\kappa_1, \quad\text{ for all }0<t<r_0,
\quad\text{ and }\quad \frac{\cos t -1}{t^2}\geq -\frac12,\quad\text{ for all }t>0.
$$
Plugging these estimates in \eqref{eqn:Ngleich3} yields the assertion (i) for $N=3$ with $\kappa_2=\frac1{8\pi}$.

To prove the assertion (ii), we notice that for $\alpha\in\N_0^N$ and $k=|\alpha|$, 
an induction argument based on the recursion formula
$\frac{d}{dt}\left[t^{-\nu}Y_{\nu}(t)\right] = -t^{-\nu}Y_{\nu +1}(t)$ (see \cite[p.105]{Lebedev2012})
gives
$$
\partial^\alpha(\Psi(x)-\Lambda(x)) = \sum_{\ell=0}^{\lfloor\frac{k}{2}\rfloor}f_{k-\ell}(|x|) P_{k-2\ell}(x),
$$
where for $m\in\N_0$, $P_m(x)$ is a homogeneous polynomial of degree $m$ and where
$$
f_m(t)=(-1)^m\frac{2^m\Gamma(\frac{N-2}{2}+m)}{\Gamma(\frac{N-2}{2})}
\frac{t^{2-N-2m}}{N(N-2)\omega_N}\Bigl[\eta_{\frac{N-2}{2}+m}(t)-1\Bigr], 
\quad t>0.
$$
As a consequence, given $r>0$, there is a constant $\gamma=\gamma(N,k,r)>0$ such that
$$
\bigl|\partial^\alpha(\Psi(x)-\Lambda(x))\bigr| 
\leq \gamma|x|^{2-N-k}\sum_{\ell=0}^{\lfloor\frac{k}{2}\rfloor}\Bigl[\eta_{\frac{N-2}{2}+k-\ell}(|x|)-1\Bigr],
\quad\text{for all }|x|<r.
$$
Using \eqref{eqn:eta_above} and remarking that $\frac{N-2}{2}+k-\ell>1$ for $k\geq 1$ 
and $0\leq \ell\leq \lfloor\frac{k}{2}\rfloor$, 
we obtain the desired assertion and the lemma is proven.
\end{proof}

\subsection{Compactness properties}\label{sec:compact}
Here, and in the next section, we discuss properties of the resolvent Helmholtz operator $\textbf{R}:=\textbf{R}_{1}$ given
by the convolution $f\mapsto \Psi\ast f$ for $f\in\cS(\RN)$, where $\Psi$ is given in \eqref{eqn:realteil}. 
Let us first remark that as a consequence of an estimate of Kenig, Ruiz and Sogge \cite[Theorem 2.3]{kenig1987},
this mapping extends as a continuous linear operator
$$
\mathbf{R}\ :\ L^{\crid}(\RN)\ \to\ L^{\crit}(\RN).
$$
In particular, there exists a constant $C_0>0$ only depending on $N$ such that
\begin{equation}\label{eqn:c0}
\|\mathbf{R}v\|_{\crit}\leq C_0\|v\|_{\crid}, \quad\text{for all }v\in L^{\crid}(\RN).
\end{equation}
Let us denote by
$$
\mathbf{R}_0\ :\ L^{\crid}(\RN)\ \to\ L^{\crit}(\RN)
$$
the linear operator given by the convolution with the fundamental solution of Laplace's equation
$$
\mathbf{R}_0v:=\Lambda\ast v, \quad v\in L^{\crid}(\RN).
$$
Notice that $\mathbf{R}_0$ is well defined and continuous, as a consequence 
of the weak Young inequality~\cite[p.~107]{Lieb_Loss2001}.

\begin{remark}
The results in this and the next sections are stated and proven for the real part $\mathbf{R}$ of the resolvent, but
they remain valid for the full resolvent $\cR$: $L^{\crid}(\RN)$ $\to$ $L^{\crit}(\RN)$ 
which is the extension of the convolution map $f\mapsto \Phi\ast f$, $f\in\cS(\RN,\C)$.
\end{remark}
\begin{lemma}\label{lem:AQsubkritischkompakt}
For all $1 \leq t < \crit$ and all $r>0$ the operator 
$\mathds{1}_{B_{r}}\mathbf{R}$ : $L^{\crid}(\RN)$ $\rightarrow$ $L^{t}(\RN)$ is compact.
\end{lemma}
\begin{proof}
By elliptic estimates (see \cite[Proposition A.1]{Evequoz2015}), we can find for every $r>0$ a constant $D_r>0$ such that 
$\|\textbf{R}v\|_{W^{2,\crid}(B_r)}\leq D_r \|v\|_{\crid}$ for all $v \in L^{\crid}(\RN)$.
Since the embedding $W^{2,\crid}(B_r) \hookrightarrow L^{t}(B_r)$ is compact for all $1 \leq t < \crit$, and all $r>0$,
we deduce that the operator $ \mathds{1}_{B_{r}}\textbf{R} : L^{\crid}(\RN) \rightarrow L^{t}(\RN)$ 
is compact for all $1 \leq t < \crit$ and all $r>0$.
\end{proof}
\begin{proposition}\label{prop:lokkomp}
	~
\begin{itemize}
	\item[(i)] The difference  $(\mathbf{R}-\mathbf{R}_0)$ is a continuous linear mapping from $ L^{\crid}(\RN)$ 
	into $ W^{2,\crit}(\RN) $
	\item[(ii)]	
For all $r>0$, the operator \
$\mathds{1}_{B_r}(\mathbf{R}-\mathbf{R}_0)$: $L^{\crid}(\RN)$ $\to$ $L^{\crit}(\RN)$ \
is compact.
	\end{itemize}
\end{proposition}
\begin{proof}
In the sequel, for $\mu\in\R$, $C$ and $C_\mu$ shall denote constants depending on $N$ and on 
$N, \mu$ respectively, but which may change from line to line.

To prove (i) we shall use a decomposition of $\Psi$, 
similar to the one introduced in \cite[Section 3]{Evequoz2015} for $\Phi$. 
We fix a radial $\psi \in \SR(\RN)$ such that $\widehat{\psi} \in \ceinftyc(\RN)$, $0 \leq \widehat{\psi}\leq 1$, 
$\widehat{\psi}(\xi) = 1$ on for $\bigl| |\xi| - 1 \bigr| \leq \frac{1}{6}$ and $\widehat{\Psi}(\xi) = 0$ for 
$\bigl| |\xi| - 1 | \geq \frac{1}{4}$. 
Write $\Psi = \Psi_{1} + \Psi_{2}$ with
\begin{equation}\label{eqn:decomp_Psi}
\Psi_{1}:= (2\pi)^{-\frac{N}{2}} (\Psi \ast \psi), \qquad \Psi_{2} = \Psi - \Psi_{1}.
\end{equation}
Then, for every $f\in\cS(\RN)$ and $\alpha\in\N^N_0$, the properties of the convolution of Schwartz functions with
a tempered distribution (see \cite[Theorem 7.19]{Rudin_FA}) allow to write
$$
(\partial^\alpha\Psi_1)\ast f =(2\pi)^{-\frac{N}{2}}  [\Psi\ast(\partial^\alpha\psi)]\ast f
=(2\pi)^{-\frac{N}{2}} \Psi\ast[(\partial^\alpha\psi)\ast f],
$$
where $\partial^\alpha\psi\in\cS(\RN)$. Hence, from \eqref{eqn:c0} and Young's inequality for the convolution,
we obtain the estimate
$$
\left\|(\partial^\alpha\Psi_1)\ast f\right\|_{\crit}=(2\pi)^{-\frac{N}{2}} \left\|\Psi\ast[(\partial^\alpha\psi)\ast f]\right\|_{\crit}
\leq(2\pi)^{-\frac{N}{2}}  C_0 \|\partial^\alpha\psi\|_1 \|f\|_{\crid}, \quad\text{for all }f\in\cS(\RN).
$$
As a consequence, the convolution
$f \mapsto (\partial^{\alpha}\Psi_{1})\ast f$, $f\in\cS(\RN)$, 
extends as a continuous map from $L^{\crid}(\RN)$ into $L^{\crit}(\RN)$ for every $\alpha \in \N^{N}_{0}$.

Turning to $\Psi_2$, we have by definition $\wh{\Psi_{2}} = (1- \wh{\psi})\wh{\Psi} $ and, since 
taking real parts in \eqref{eqn:Fourier_Phi} yields 
$$
\wh{\Psi_2}(\xi) = (2\pi)^{-\frac{N}2}\lim\limits_{\eps \to 0^+}\frac{|\xi|^2-1}{(|\xi|^2 - 1)^2 +\eps^2}\left(1-\wh{\psi}(x)\right)
=(2\pi)^{-\frac{N}2}\frac{\left(1-\wh{\psi}(x)\right)}{|\xi|^2 - 1},
$$
we get $\wh{\Psi_{2}} \in \ceinfty(\RN)$ and $\wh{\Psi_{2}}(\xi) =(2\pi)^{-\frac{N}2} (|\xi|^2 - 1)^{-1}$ for $|\xi|\geq \frac{5}{4}$. 
This gives $\partial^{\beta}\widehat{\Psi_{2}} \in L^{1}(\RN)$ for all $\beta \in \N^{N}_{0}$ such that $2 + |\beta| > N$. 
Therefore, using standard differentiation properties of the Fourier transform, 
the fact that $\widehat{\Psi_{2}}$ (and so $\Psi_{2}$) is radial 
and that ${\mathcal{F}}(f)(\xi) = {\mathcal{F}}^{-1}(f)(-\xi)$, we obtain
$$ 
\Bigl\| |\cdot|^{|\beta|}\Psi_{2}\Bigr\|_{\infty}
= \left\|{\cal{F}}\left(\partial^{\beta}\widehat{\Psi_{2}}\right)\right\|_{\infty} 
\leq \left\|\partial^{\beta}\widehat{\Psi_{2}}\right\|_{L^{1}(\RN)} 
\leq C_{|\beta|}, \quad \text{for all } |\beta|>N-2. 
$$
Choosing $\beta\in\N_0^N$ with $|\beta|=N$, we obtain that
\begin{equation}\label{eqn:psi2}
|\Psi_{2}(x)| \leq C|x|^{-N}, \qquad  \text{ for all } x\in\RN.
\end{equation}
Using the same argument with $\partial^{\alpha}\Psi_{2}$ in place of $\Psi_{2}$, for every $\alpha \in \N^{N}_{0}$, we get
\begin{equation}\label{eqn:psi2ableitung}
|\partial^{\alpha}\Psi_{2}(x)| \leq C_{|\alpha|}|x|^{-N-|\alpha|}, \quad \text{for all } x \in \RN 
\text{ and all } \alpha \in \N^{N}_{0}.
\end{equation}
From Lemma \ref{lem:difference_fundsol}, we obtain estimates on $\partial^{\alpha}(\Psi_{2} - \Lambda)(x)$ for $|x|$ small. 
For large values of $|x|$, we use \eqref{eqn:psi2ableitung} and  $|\partial^{\alpha}\Lambda(x)| \leq C_{|\alpha|}|x|^{2-N-|\alpha|}$, 
which follows easily from \eqref{eqn:laplacefundsol}. Altogether, we get for $\alpha\in\N^N_0$ and $x\in\RN$,
\begin{equation}
	|\partial^{\alpha}(\Psi_{2} - \Lambda)(x)| \leq 
	\begin{cases}
	C_{|\alpha|}\min \{ |x|^{4-N-|\alpha|},|x|^{2-N-|\alpha|} \}, \quad & N=3,\ N\geq 5,\text{ or }N=4\text{ and }|\alpha|\geq 1, \\
	C \min \{ 1+\bigl|\ln|x|\bigr|, |x|^{-2} \}, \quad &N=4\text{ and }|\alpha|=0.
	\end{cases}
\end{equation}
As a consequence, denoting by $L^{\frac{N}{N-2}}_{w}(\RN)$ the weak-$L^{\frac{N}{N-2}}$ space, we infer that
$$ 
\partial^{\alpha}(\Psi_{2} - \Lambda) \in L^{\frac{N}{N-2}}_{w}(\RN), \quad 
\text{for all } \alpha \in \N^{N}_{0} \text{ such that } |\alpha|\leq 2.
$$
From the weak Young inequality, the convolution 
$f \mapsto \partial^\alpha(\Psi_{2} - \Lambda)\ast f$, $f\in\cS(\RN)$, 
extends as a continuous map from $L^{\crid}(\RN)$ into $L^{\crit}(\RN)$ for such $\alpha$.
Summarizing and using the fact that
$$
\left\|(\Psi-\Lambda)\ast f\right\|_{W^{2,\crit}}^2
\leq 2\sum\limits_{|\alpha|\leq 2} \left\| \partial^{\alpha}(\Psi_{2} - \Lambda)\ast f\right\|^{2}_{\crit}  
+ 2\sum\limits_{|\alpha|\leq 2} \left\| (\partial^{\alpha}\Psi_{1})\ast f\right\|^{2}_{\crit}, \quad f\in\cS(\RN),
$$
we obtain that the convolution $f \mapsto (\Psi - \Lambda)\ast f$ 
extends as a continuous map from $L^{\crid}(\RN)$ into $W^{2,\crit}(\RN)$.	 
Therefore, the operator
$$ 
\textbf{R}-\textbf{R}_{0} :  L^{\crid}(\RN)\rightarrow W^{2,\crit}(\RN)
$$
is continuous and (i) is proven.\\
\noindent
By the Rellich-Kondrachov Theorem, the embedding $W^{2,\crit}_{loc}(\RN) \hookrightarrow L^{t}_{loc}(\RN)$ is compact 
for all $1 \leq t < \frac{2N}{(N-6)_{+}}$.
Thus, we obtain the compactness of 
$\mathds{1}_{B_{r}}(\mathbf{R}-\mathbf{R}_0)$ : $L^{\crid}(\RN)$ $\to$ $L^{\crit}(\RN)$ for all $r> 0$, which proves (ii).
\end{proof}

\subsection{Nonvanishing property and related estimates}
As a key ingredient for the existence result in Section~\ref{sec:dual} below, 
we prove that the nonvanishing property of the quadratic form associated with the Helmholtz resolvent
holds true in the space $L^{p'}(\RN)$ with $p=2^\ast$. This property has been proved in \cite[Theorem 3.1]{Evequoz2015} 
in the noncritical range $\frac{2(N+1)}{N-1}<p<2^\ast$.

\begin{theorem}\label{nonvanishing}
Consider a bounded sequence $(v_n)_n\subset L^{2^+}(\R^N)$
satisfying 
\begin{equation}\label{eqn:hyp_nonvanish}
\limsup \limits_{n\to\infty}\left|\int_{\R^N}v_n\mathbf{R} v_n\, dx\right|>0.
\end{equation}
Then there exists $R>0$, $\zeta>0$ and a sequence
$(x_n)_n\subset\R^N$ such that, up to a subsequence, 
\begin{equation}
 \int_{B_R(x_n)}|v_n|^{2^+}\, dx \geq \zeta,\quad\text{for all }n.
\end{equation}
\end{theorem}

\begin{proof}
Consider the decomposition \eqref{eqn:decomp_Psi} of $\Psi$ introduced in the proof of Proposition~\ref{prop:lokkomp}.

We start by looking at the case where $v_n\in\cS(\RN)$ for all $n$, and we assume by contradiction that
\begin{equation}\label{hyp:vanishing}
\lim_{n\to\infty}\left(\sup_{y\in\RN}\int\limits_{B_\rho(y)}|v_n|^{2^+}\dx\right)=0, 
\quad \text{for all }\rho>0.
\end{equation}
Since Lemma 3.4 in \cite{Evequoz2015} holds for the critical exponent $p=2^\ast$,
we obtain that
\begin{equation}\label{eqn:Lemma3.4}
\int\limits_{\RN} v_n [\Psi_1 \ast v_n] \dx \to 0, \quad \text{as }n \to \infty,
\end{equation}
taking real parts.
Turning to $\Psi_2$, we note that the estimate \eqref{eqn:psi2} and the behaviour of $\Psi$ 
close to $x=0$ given by \eqref{asympt:Psi}
yield the existence of some constant $C'=C'(N)>0$ such that
\begin{equation}\label{eqn:phi2}
|\Psi_2(x)|\leq C' \min\{|x|^{2-N},|x|^{-N}\}, \qquad \text{for all }x \neq 0.
\end{equation}
Setting $M_R:= \R^N\backslash B_R$ for $R >1$, we deduce from \eqref{eqn:phi2} that  
\begin{equation*}
\|\Psi_2\|_{L^{\frac{N}{N-2}}(M_R)}
\leq C'\left(\ \int\limits_{|x|\geq R}|x|^{-\frac{N^2}{N-2}}\, dx\right)^\frac{N-2}{N} \to 0, \qquad \text{as }R \to \infty.  
\end{equation*}
Hence, by Young's inequality,
\begin{equation}\label{eqn:phi2_outside}
\sup_{n \in \N} \Bigl|\ \int\limits_{\RN} v_n [(\mathds{1}_{M_R} \Psi_2)*v_n]\,dx\Bigr| 
\leq \|\Psi_2\|_{L^{\frac{N}{N-2}}(M_R)} \sup_{n \in \N} \|v_n\|_{L^{2^+}(\R^N)}^2 \to 0,
\quad \text{as }R \to \infty. 
\end{equation}
Consider a decomposition of $\RN$ into disjoint $N$-cubes $\{Q_\ell\}_{\ell\in\N}$ 
of side length $R$, and let for each $\ell$ the $N$-cube $Q'_\ell$ have
the same center as $Q_\ell$ but side length $3R$. From the estimate \eqref{eqn:phi2}
and the Hardy-Littlewood-Sobolev inequality~\cite[Theorem 4.3]{Lieb_Loss2001}, 
there is a constant $C''=C''(N)$ such that
\begin{align*}
\Bigl|\int\limits_{\RN} v_n [(\mathds{1}_{B_R} \Psi_2)\ast v_n] \dx \Bigr|
 &\leq \sum_{\ell=1}^\infty\int\limits_{Q_\ell}\Bigl(\int\limits_{|x-y|<R}|\Psi_2(x-y)\,||v_n(x)|\, |v_n(y)| \dy\Bigr) dx\\
 &\leq
 C' \sum_{\ell=1}^\infty\int\limits_{Q_\ell}\left(\int\limits_{Q'_\ell}\frac{|v_n(x)|\,  |v_n(y)|}{|x-y|^{N-2}} \dy\right) dx\\ 
&\leq C'' \sum_{\ell=1}^\infty\Bigl(\int\limits_{Q_\ell'}|v_n(x)|^{\crid} \dx\Bigr)^{\frac{N+2}{N}}
\\
 & \leq C'' \Bigl[\:\sup_{\ell \in \N}\int_{Q'_\ell}|v_n(x)|^{\crid} \dx\Bigr]^{\frac2N} 
 \:\sum_{\ell=1}^\infty\int_{Q'_\ell}|v_n(x)|^{\crid} \dx\\ 
&\leq C''\Bigl[\:\sup_{y \in \RN}\int_{B_{3R\sqrt{N}}(y)}|v_n(x)|^{\crid} \dx\Bigr]^{\frac2N}
\: 3^N \|v_n\|_{\crid}^{\crid}, \qquad \text{for all } n.
\end{align*}
Therefore, the boundedness of $(v_n)_n$ and the assumption \eqref{hyp:vanishing} give
\begin{equation} \label{eqn:phi2_inside}
\lim_{n \to \infty} \int\limits_{\RN} v_n [(\mathds{1}_{B_R} \Psi_2)\ast v_n] \dx= 0, \qquad \text{for every } R>0.
\end{equation}
Combining \eqref{eqn:Lemma3.4}, \eqref{eqn:phi2_outside} and \eqref{eqn:phi2_inside}, we obtain
$$
\int\limits_{\RN} v_n \mR v_n \dx=\int\limits_{\RN} v_n [\Psi_1 \ast v_n] \dx+\int\limits_{\RN} v_n [\Psi_2 \ast v_n] \dx \to 0, 
\quad \text{as }n \to \infty,
$$
contradicting the assumption \eqref{eqn:hyp_nonvanish}.
Arguing by density, as in the proof of \cite[Theorem 3.1]{Evequoz2015}, the theorem follows.
\end{proof}
Let us recall a result obtained recently \cite[Lemma 2.4]{Evequoz_ampa2017}, 
on the bilinear form associated to the operator $\mR$ for functions having disjoint support.
Its proof relies on resolvent estimates involving the decomposition \eqref{eqn:decomp_Psi} of the fundamental
solution of the Helmholtz equation and on the decay bound in \cite[Proposition 3.3]{Evequoz2015}.
\begin{lemma}\label{lem:interact}
Let $p>\frac{2(N+1)}{N-1}$.
There exists a constant $D=D(N,p)>0$ such that 
for any $R>0$, $r\geq 1$ and $u, v\in L^{p'}(\R^N)$ 
with $\text{supp}(u)\subset B_R$
and $\text{supp}(v)\subset\R^N\backslash B_{R+r}$,
\begin{equation*}
\left|\int_{\R^N}u\mathbf{R} v\, dx\right|\leq D r^{-\lambda_p}\|u\|_{p'}\|v\|_{p'}, 
\quad\text{ where }\ \lambda_p=\frac{N-1}{2}-\frac{N+1}{p}.
\end{equation*}
\end{lemma}
Based on this estimate, we prove a technical result which will be used in Section~\ref{sect:PS} 
below to deal with a remainder term in an estimate derived from the Hardy-Littlewood-Sobolev inequality.
\begin{lemma}\label{lem:technisch}
Let $(z_n)_n\subset L^{\crid}(\RN)$ be a bounded sequence. Then, for every $\eps >0$, there exists
$\rho_\eps>0$ such that
$$
\liminf_{n\to\infty}\left|\ \int\limits_{\RN}\mathds{1}_{B_\rho}z_n \mathbf{R}\bigl(\mathds{1}_{M_\rho}z_n\bigr)\dx\right|<\eps,
\quad\text{for all }\rho\geq\rho_\eps.
$$
Here, $M_\rho:=\RN\backslash B_\rho$.
\end{lemma}
\begin{proof}
Let $\zeta:=\sup\{\|z_n\|_{\crid}\, :\, n\in\N\}$.
We first see that by Lemma~\ref{lem:interact} there is a constant $D=D(N)>0$ such that
$$
\left|\ \int\limits_{\RN}\mathds{1}_{B_\rho}z_n \textbf{R}\bigl(\mathds{1}_{M_{2\rho}}z_n\bigr)\dx\right|
\leq D\zeta^2\rho^{-\frac1N}, \quad\text{ for all }n\in\N\text{ and every }\rho\geq 1.
$$
Hence, setting $\rho_0:=\max\bigl\{1, \left(\frac{2D\zeta^2}{\eps}\right)^N\bigr\}$ we find
$$
\left|\ \int\limits_{\RN}\mathds{1}_{B_\rho}z_n \textbf{R}\bigl(\mathds{1}_{M_{2\rho}}z_n\bigr)\dx\right|\leq \frac{\eps}2,
\quad\text{for all }n\in\N\text{ and every } \rho\geq\rho_0.
$$
Next, we choose $\eta>0$ such that $\eta<\left(\frac{\eps}{2C_0\zeta}\right)^{\crid}$, 
where $C_0>0$ is such that \eqref{eqn:c0} holds, and we claim that 
\begin{equation}\label{eqn:Absch_Ringe}
\text{there exists }\rho_1>0 \text{ such that}\quad 
\liminf_{n\to\infty}\int_{B_{2\rho}\backslash B_{\rho}}|z_n|^{\crid}\dx<\eta, \quad\text{for all }\rho\geq\rho_1.
\end{equation}
Suppose this is not the case. Then, for every $k\in\N$ we can find a radius $\rho_k\geq k$ 
and an index $n_0(k)\in\N$ for which
$$
\int_{B_{2\rho_k}\backslash B_{\rho_k}}|z_n|^{\crid}\dx\geq\eta,\quad\text{for all }n\geq n_0(k).
$$
Moreover, we can assume without loss of generality that $n_0(k+1)\geq n_0(k)$ and $\rho_{k+1}\geq 2\rho_k$. 
For each $\ell\in\N$, it follows that
$$
\zeta^{\crid}\geq \int_{\RN}|z_n|^{\crid}\, dx\geq \sum_{k=1}^{\ell}\int_{B_{2\rho_k}\backslash B_{\rho_k}}|z_n|^{\crid}\dx
\geq\ell\eta,\quad\text{for all }n\geq n_0(\ell).
$$
For $\ell$ large enough we obtain a contradiction, and the claim is proven.

As a consequence of the above results, we can write for $\rho\geq\rho_\eps:=\max\{\rho_0,\rho_1\}$,
\begin{align*}
\left|\ \int\limits_{\RN}\mathds{1}_{B_\rho}z_n \textbf{R}\bigl(\mathds{1}_{M_\rho}z_n\bigr)\dx\right|
&\leq \left|\ \int\limits_{\RN}\mathds{1}_{B_\rho}z_n \textbf{R}\bigl(\mathds{1}_{M_{2\rho}}z_n\bigr)\dx\right|
+ \left|\ \int\limits_{\RN}\mathds{1}_{B_\rho}z_n \textbf{R}\bigl(\mathds{1}_{B_{2\rho}\backslash B_{\rho}}z_n\bigr)\dx\right|\\
&\leq \frac{\eps}2 + C_0 \zeta \|\mathds{1}_{B_{2\rho}\backslash B_{\rho}}z_n\|_{\crid},
\end{align*}
using H\"older's inequality and the estimate \eqref{eqn:c0}.
The conclusion then follows from the claim~\eqref{eqn:Absch_Ringe}.
\end{proof}
%
\section{Existence via the dual variational method}\label{sec:dual}
\subsection{The dual energy functional}
We follow the path established in \cite{Evequoz2015}  
and use a dual variational framework to find nontrivial solutions for the problem
\begin{equation}\label{problem}
- \Delta u -u = Q(x)|u|^{2^{\ast} -2}u, \quad u\in W^{2,\crit}(\RN),
\end{equation}
where $Q  \in  L^{\infty}(\R^{N}) $ is a nonnegative function which is not identically zero. 
Setting $v = Q^{\frac{1}{\crid}}|u|^{\crit - 2}u$, we shall study the fixed-point problem
\begin{equation}\label{eqn:dualelsng} 
|v|^{\crid - 2}v = Q^{\frac{1}{\crit}}\textbf{R}(Q^{\frac{1}{\crit}}v), \qquad  v\in L^\crid(\RN), 
\end{equation}
where $\textbf{R}$ denotes the resolvent Helmholtz operator defined in Section~\ref{sec:compact}.
For the Birman-Schwinger type operators associated to the Helmholtz and Laplace resolvents respectively, we introduce
the notation 
\begin{equation}\label{eqn:BS}
\textbf{A}_{Q}v := Q^{\frac{1}{\crit}}\textbf{R}(Q^{\frac{1}{\crit}}v)\quad
\text{and}\quad \mathbf{G}_Qv := Q^{\frac{1}{\crit}}\textbf{R}_0(Q^{\frac{1}{\crit}}v),
\quad v\in L^{2^+}(\RN).
\end{equation}
We consider the functional
\begin{equation}\label{functional}
J_{Q}(v) := \frac{1}{\crid}\int\limits_{\R^{N}}|v|^{\crid} \dx - \frac{1}{2}\int\limits_{\R^{N}}v\textbf{A}_{Q}v \dx,
\quad\text{for }v \in L^{\crid}(\R^{N}).
\end{equation}
It is known that $J \in \ce^{1}(L^{\crid}(\RN),\R)$ and from the symmetry of $\textbf{A}_{Q}$ 
(cf.~\cite[Lemma 4.1]{Evequoz2015}) we have
$$ 
J_{Q}'(v)w = \int\limits_{\RN}\left(|v|^{\crid - 2}v - \textbf{A}_{Q}v\right)w \dx, 
\quad \text{for all } v,w \in L^{\crid}(\RN). 
$$

We detect solutions of \eqref{problem} by finding critical points of the functional $J_Q$. 
Indeed, for $v \in L^{\crid}(\RN)$ we have $J_Q'(v) = 0$ if and only if $v$ satisfies \eqref{eqn:dualelsng}.
Setting $u  = \textbf{R}(Q^{\frac{1}{\crit}}v)$, we see that $u$ solves $u = \textbf{R}\left(Q|u|^{\crit - 2}u\right)$.
Any solution of this integral equations has the following properties:
\begin{lemma}[{Special case of \cite[Lemma 4.3]{Evequoz2015}}]\label{lem:4.3}
Let $Q \in L^{\infty}(\RN)$ and consider a solution $u \in L^{\crit}(\RN)$ of $u = \textbf{R}\left(Q|u|^{\crit - 2}u\right)$. 
Then, $u\in W^{2,q}(\RN)$ for all $\crit\leq q< \infty$ and $u$ is a strong solution of \eqref{problem}. 
Moreover, $u$ is the real part of a function $\wt{u}$ which satisfies Sommerfeld's outgoing radiation condition
in the form
\begin{equation*}
\lim\limits_{R \to \infty}\frac1R\int\limits_{B_R}\left| \nabla \wt{u}(x) - i \wt{u}(x) \frac{x}{|x|}\right|^2\,dx = 0.
\end{equation*} 
In addition, $u$ satisfies the following asymptotic relation
\begin{equation*}
\lim\limits_{R \to \infty}\frac1R\int\limits_{B_R} \left| u(x) - \sqrt{\frac{\pi}{2}}\ 
\text{Re}\left[\frac{e^{i|x| - \frac{i(N-3)\pi}{4}}}{|x|^{\frac{N-1}{2}}}
\mathcal{F}\left(Q|u|^{\crit-2}u\right)\left(\frac{x}{|x|}\right)\right] \right|^2\,dx = 0.
\end{equation*}
\end{lemma}
As shown in \cite[Lemma 4.2]{Evequoz2015}, the functional $J_Q$ has the mountain pass geometry, i.e.
\begin{align*}
&(\text{MP1}) \text{ there exists } \delta>0\text{ and }\rho>0\text{ such that }J_Q(v)\geq\delta>0, 
\text{ for all }v\in L^{\crid}(\RN)\text{ with }\|v\|_{\crid}=\rho;\\
&(\text{MP2}) \text{ there exists } v_0\in L^{\crid}(\RN)\text{ such that }\|v_0\|_{\crid}>\rho\text{ and }J_Q(v_0)<0.
\end{align*}
The mountain pass level
$$
L_Q:=\inf\limits_{P\in \mathcal{P}}\max\limits_{t\in[0,1]}J_Q(P(t)), \quad
\text{where}\quad \mathcal{P}=\left\{P\in \ce([0,1],L^{\crid}(\R^N))\, :\, P(0)=0
\text{ and }J_Q(P(1))<0\right\},
$$
is therefore well defined, $0<L_Q<\infty$, and by the same arguments as in \cite[Lemma 4.1]{Evequoz_na2017},
it can be characterized as the following infimum
\begin{equation}\label{eqn:MP_equivalent}
\begin{aligned}
L_Q&=\inf\left\{J_Q(t_vv)\, :\, v\in L^{\crid}(\RN) \text{ with } \int\limits_{\RN}v\mathbf{A}_Q v \dx>0\right\}\\
&=\inf\left\{\frac{1}{N}\left( \frac{\left\|v\right\|^{2}_{\crid}}{\int\limits_{\RN}v\textbf{A}_{Q}v \dx}\right)^{\frac{N}{2}}
\, :\, v\in L^{\crid}(\RN) \text{ with } \int\limits_{\RN}v\mathbf{A}_Qv \dx>0\right\}.
\end{aligned}
\end{equation}
Here, for $v\in L^{\crid}(\RN)$ with $\int\limits_{\RN}v\mathbf{A}_Qv \dx>0$,
\begin{equation}\label{zeit}
t_{v}=\left(\frac{\int\limits_{\RN}|v|^{\crid}~dx}{\int\limits_{\RN}v\textbf{A}_{Q}v~dx}\right)^{\frac{1}{2-\crid}}
\end{equation} 
denotes the unique $t>0$ with the property $J_Q(t_vv)=\max\limits_{t>0}J_Q(tv)$. 
Remarking that for every such $v$, $J_Q'(t_vv)v=0$, we see that if $L_Q$ is achieved by some critical point of $J_Q$, 
then $L_Q$ coincides with the least-energy level, i.e.,
$$
L_Q=\inf\left\{J_Q(v)\, :\,  v\in L^{\crid}(\RN)\backslash\{0\}\text{ with }J'_Q(v)=0\right\}.
$$
Following the terminology introduced in \cite{Evequoz_ampa2017}, we will call a solution $u$ of the nonlinear Helmholtz
equation \eqref{problem} a {\em dual ground state}, if $u  = \textbf{R}(Q^{\frac{1}{\crit}}v)$ and $v\in L^{\crid}(\RN)$ is a
critical point of the functional $J_Q$ at the mountain pass level, i.e., $J_Q'(v)=0$ and $J_Q(v)=L_Q$. As a consequence
of the discussion at the beginning of the section, every dual ground state $u$ has the properties stated 
in Lemma~\ref{lem:4.3}.
\subsection{Palais-Smale sequences}\label{sect:PS}
In this section, we investigate the properties of Palais-Smale sequences for the functional $J_Q$.
Recall that a sequence $(v_n)_n\subset L^{\crid}(\RN)$ is called a {\em Palais-Smale sequence} for $J_Q$ if $(J_Q(v_n))_n$
is bounded and $\|J_Q'(v_n)\|_\ast \to 0$ as $n\to\infty$. Here, $\|\cdot\|_\ast$ denotes the dual norm to $\|\cdot\|_{\crid}$. 
If, in addition, $J_Q(v_n)\to \beta$ as $n\to\infty$ for some $\beta\in\R$,
$(v_n)_n$ is called a {\em (PS)$_\beta$-sequence} for $J_Q$. We start by considering 
sequences which satisfy a localized version of
the above property. For this purpose, we introduce the following piece of notation: If $v\in L^{\crid}(\RN)$, 
we let $J_{Q}'(v)\mathds{1}_{B_r}$ denote the continuous linear form 
$w \mapsto J_{Q}'(v)(\mathds{1}_{B_r}w)$ on $L^{\crid}(\RN)$.
\begin{lemma}\label{lem:PS_Folgen}
Let $(v_{n})_{n} \subset L^{\crid}(\RN)$ be a bounded sequence such that, for all $r>0$, 
$\|J_{Q}'(v_n)\mathds{1}_{B_r}\|_\ast\to 0$ as $n\to\infty$. 
Then, up to a subsequence,
\begin{enumerate}[(i)]
	\item $(v_{n})_{n}$ has a weak limit $v \in L^{\crid}(\RN)$.
	\item  For all $1 \leq q < \crid$ and all $r>0$, we have $\mathds{1}_{B_r}v_{n} \rightarrow \mathds{1}_{B_r}v$, strongly
	 in $L^{q}(\RN)$ and $v_n\to v$ a.e. on $\RN$, as $n \rightarrow \infty$.
	\item $J_{Q}'(v) = 0$.
	\item As $n \rightarrow \infty$, we have for all $r>0$,
	\begin{equation}\label{eqn:BLnormintegral}
	\left\| \mathds{1}_{B_{r}}(v_{n}-v)\right\|^{\crid}_{\crid} 
	= \int\limits_{\RN}\mathds{1}_{B_{r}}(v_{n} - v) \textbf{A}_{Q}(v_{n} - v) ~\dx + o(1).
	\end{equation}	
\end{enumerate}	
\end{lemma}
\begin{proof}
Since $(v_n)_n$ is bounded in $L^{\crid}(\RN)$, there exists $v\in L^{\crid}(\RN)$ and a subsequence
which we still denote by $(v_n)_n$ such that $v_n\weakto v$ weakly. This proves (i). From now on, we restrict
to this particular subsequence. To prove (ii), let $r>0$, $1 \leq t < \crit$ and $\varphi \in \ceinftyc(\RN)$. 
For $n,m \in \N$ we have 
\begin{align*}
	&\left|\  \int\limits_{\RN}  \left(  \mathds{1}_{B_r}|v_{n}|^{\crid-2}v_{n} 
	- \mathds{1}_{B_r}|v_{m}|^{\crid - 2}v_{m}  \right) \varphi \dx \right|\\
	&\qquad=
	\left| \Bigl[J_{Q}'(v_{n}) - J_{Q}'(v_{m})\Bigr](\mathds{1}_{B_r}\varphi)
	+ \int\limits_{\RN}\mathds{1}_{B_r}\varphi \textbf{A}_{Q}(v_{n} - v_{m}) \dx \right| \\
	&\qquad\leq
	\left\| J_{Q}'(v_{n})\mathds{1}_{B_r} - J_{Q}'(v_{m})\mathds{1}_{B_r} \right\|_\ast 
	\left\|\mathds{1}_{B_r}\varphi\right\|_{\crid} 
	+ \left\| \mathds{1}_{B_r}\textbf{A}_{Q}(v_{n} - v_{m})\right\|_{t}  \left\|\varphi\right\|_{t'} \\
	&\qquad\leq
	C \left[  \left\| J_{Q}'(v_{n})\mathds{1}_{B_r} \right\|_\ast + \left\| J_{Q}'(v_{m})\mathds{1}_{B_r}\right \|_\ast \right] 
	\left\|\varphi\right\|_{t'} 
	+ \|Q\|_\infty^{\frac1\crit}\left\|\mathds{1}_{B_r}\textbf{R}\left(Q^{\frac{1}{\crit}}\left(v_{n} - v_{m}\right)\right)\right\|_{t}
	\left\|\varphi\right\|_{t'},
\end{align*}
where the constant $C>0$ depends on $N$ and $r$.
The first expression in the last line vanishes as $n, m \rightarrow \infty$, by assumption,
and the second one also vanishes due to Lemma \ref{lem:AQsubkritischkompakt}. 
Therefore, arguing by density, we find that $(\mathds{1}_{B_r}|v_{n}|^{\crid-2}v_{n})_{n}$ is a Cauchy sequence 
in $L^{t}(\RN)$. Since the mapping $N : L^{t}(\RN) \rightarrow L^{\frac{t}{\crit-1}}(\RN)$ given by 
$N(u):= |u|^{\crit - 2} u$ is well defined and Lipschitz continuous, it follows that
$(\mathds{1}_{B_r}v_{n})_{n}=\bigl(N(\mathds{1}_{B_r}|v_{n}|^{\crid-2}v_{n})\bigr)_{n}$ is a Cauchy sequence in $L^q(\RN)$
for all $1\leq q<\crid$. Since these spaces are complete, and since $\mathds{1}_{B_r}v_n\weakto \mathds{1}_{B_r}v$ 
in each of these spaces, we obtain the desired strong convergence 
$\mathds{1}_{B_r}v_n\to\mathds{1}_{B_r}v$ in $L^q(\RN)$ for all $1\leq q<\crid$.
Going to a subsequence, we also have the pointwise convergence $v_n(x)\to v(x)$ as $n\to\infty$, 
for almost every $x\in\R^N$.
	
Assertion (iii) now follows from (ii),
since for $\varphi \in \ceinftyc(\RN)$ and $r>0$ such that $\text{supp}(\varphi)\subset B_r$, we have
\begin{align*}
	J_{Q}'(v)\varphi &= \int\limits_{\RN}\mathds{1}_{B_r}|v|^{\crid -2}v \varphi \dx 
	- \int\limits_{\RN}\mathds{1}_{B_r}\varphi \textbf{A}_{Q}v \dx\\
	&=
	\lim\limits_{n \rightarrow \infty}\left[\ \int\limits_{\RN}\mathds{1}_{B_r}|v_{n}|^{\crid -2}v_{n} \varphi\dx 
	- \int\limits_{\RN}\mathds{1}_{B_r}\varphi \textbf{A}_{Q}v_{n} \dx  \right]\\
	&=
	\lim\limits_{n \rightarrow \infty}J_{Q}'(v_{n})\mathds{1}_{B_{r}}\varphi = 0.
\end{align*}
To prove assertion (iv) we use the assumption $\|J_{Q}'(v_n)\mathds{1}_{B_r}\|_\ast\to 0$ as $n\to\infty$ and write
\begin{equation}\label{eqn:Gradient_lokal}
\begin{aligned}
	o(1)  &= J'_{Q}(v_{n})\mathds{1}_{B_{r}}v_{n}
	= \left\| \mathds{1}_{B_{r}}v_{n}\right\|^{\crid}_{\crid} - \int\limits_{\RN}\mathds{1}_{B_{r}}v_{n}\textbf{A}_{Q}v_{n} \dx\\
	&=
	\left\| \mathds{1}_{B_{r}}v_{n}\right\|^{\crid}_{\crid}
	-
	\int\limits_{\RN} \mathds{1}_{B_{r}}(v_{n} - v)\textbf{A}_{Q}(v_{n} - v)\dx
	-
	\int\limits_{\RN} \mathds{1}_{B_{r}}v_{n}\textbf{A}_{Q}v\dx
	-
	\int\limits_{\RN} \mathds{1}_{B_{r}}v\textbf{A}_{Q}(v_{n} - v)\dx.
\end{aligned}
\end{equation}
The last expression vanishes as $n \rightarrow \infty$, since $v_n\weakto v$ in $L^{\crid}(\R^N)$ 
and since $\textbf{A}_Q$: $L^{\crid}(\R^N)$ $\to$ $L^{\crit}(\R^N)$ is continuous.
Furthermore, using (iii), the weak convergence $v_n\weakto v$ 
and the Br\'ezis-Lieb Lemma \cite[Lemma 1.32]{Willem}, which applies due to (ii), we obtain
\begin{align*}
	\left\| \mathds{1}_{B_{r}}v_{n}\right\|^{\crid}_{\crid} - 	\int\limits_{\RN} \mathds{1}_{B_{r}}v_{n}\textbf{A}_{Q}v\dx 
	&=
	\left\| \mathds{1}_{B_{r}}v_{n}\right\|^{\crid}_{\crid} -\left( \left\| \mathds{1}_{B_{r}}v\right\|^{\crid}_{\crid} 
	- \int\limits_{\RN} \mathds{1}_{B_{r}}v\textbf{A}_{Q}v\dx\right) - \int\limits_{\RN} \mathds{1}_{B_{r}}v_{n}\textbf{A}_{Q}v\dx\\
	&=
	\left\| \mathds{1}_{B_{r}}(v_{n} -v)\right\|^{\crid}_{\crid} +o(1), \quad\text{as }n\to\infty.
\end{align*}
Substituting in \eqref{eqn:Gradient_lokal}, the desired conclusion follows.
\end{proof}
In the above proof, the fact that the operator $\mathds{1}_{B_r}\mathbf{R}$: $L^{\crid}(\RN)$ $\to$ $L^t(\RN)$ is compact
for $1\leq t<\crit$ was essential. For $t=\crit$, the compactness does not hold anymore and therefore the assertion (ii) is false in this case.
However, in view of Proposition~\ref{prop:lokkomp}, this is only caused by the noncompactness of the operator 
$\mathds{1}_{B_r}\mathbf{R}_0$: $L^{\crid}(\RN)$ $\to$ $L^{\crit}(\RN)$ associated to the fundamental solution 
of Laplace's equation. 
In the next result, we prove that local strong convergence can be restored, provided
the mountain pass level $L_Q$ lies below the threshold value given by the least-energy level $L_Q^\ast$ 
of the functional
$$
J_Q^\ast(v):=\frac1{\crid}\int\limits_{\RN} |v|^\crid \dx -\frac12\int\limits_{\RN} v\mathbf{G}_qv \dx,
$$
where, using the notation \eqref{eqn:BS}, $\mathbf{G}_qv=q^{\frac1\crit}\mathbf{R}_0(q^{\frac1\crit}v)$ for $q=\|Q\|_\infty$.
As mentioned in the introduction, this functional arises from the limit of suitable rescaling of $J_Q$. 
The least-energy level $L_Q^\ast$ can be characterized by a formula similar to \eqref{eqn:MP_equivalent}, namely
\begin{equation}\label{eqn:MP_crit}
L_Q^\ast=\inf\left\{ \frac1N\left(\frac{\|v\|_{\crid}^2}{\|Q\|_\infty^{\frac2\crit}\int\limits_{\RN}v\mathbf{R}_0v \dx}\right)^{\frac{N}2}\, :\, 
v\in L^{\crid}(\RN)\backslash\{0\} \right\}.
\end{equation}
We note incidentally that it can be expressed in terms of the optimal constant $S$ for the Sobolev inequality in $\R^N$,
\begin{equation}\label{eqn:Sobolev_inequ}
\|\nabla u\|_2^2 \geq S \|u\|_{\crit}^2, \quad \text{for all }u\in L^{\crit}(\RN) \text{ with } \nabla u\in L^2(\RN).
\end{equation}
Indeed, it is known (see \cite{Lieb83} and also \cite{Dolbeault2011}) that the Sobolev inequality is dual 
to the Hardy-Littlewood-Sobolev inequality 
\begin{equation}\label{eqn:HLS_inequ}
\int\limits_{\RN}v \mathbf{R}_0v \dx \leq S^{-1} \|v\|_{\crid}^2
\end{equation}
and that the optimal constants are inverse to each other. Hence, we obtain
\begin{equation}\label{eqn:MP_crit_sobolev}
L_Q^\ast=\frac{S^{\frac{N}2}}{N\|Q\|_\infty^{\frac{N-2}2}}.
\end{equation}
\begin{proposition}\label{prop:PS_unter_kritisch_Level}
Let $Q\in L^\infty(\R^N)\backslash\{0\}$ have the form
$Q=Q_{\text{per}}+Q_0$, for some $Q_{\text{per}}, Q_0\geq 0$ such that $Q_{\text{per}}$ is $\Z^N$-periodic 
and $Q_0(x)\to 0$ as $|x|\to\infty$.

If $(v_n)_n\subset L^{\crid}(\RN)$ is a (PS)$_\beta$--sequence for $J_Q$ such that
$\beta=L_Q<L_Q^\ast$, then
there exists $w\in L^{\crid}(\RN)$, $w\neq 0$, such that $J_Q'(w)=0$ and $J_Q(w)=L_Q$.
\end{proposition}
\begin{proof}
Since $(v_n)_n$ is a (PS)$_\beta$--sequence for $J_Q$ with $\beta>0$, 
it is bounded (see \cite[Lemma 4.2]{Evequoz2015}). 
Using Lemma~\ref{lem:PS_Folgen}, we find
$$
\lim_{n\to\infty}\int_{\RN}Q^{\frac1{\crit}}v_n \textbf{R}\bigl(Q^{\frac1{\crit}}v_n\bigr)\dx
=\left(\frac1{\crid}-\frac12\right)^{-1}\lim_{n\to\infty}\left[ J_Q(v_n)-\frac{1}{\crid}J'_Q(v_n)v_n\right]=N\beta>0.
$$
Hence, the nonvanishing theorem, Theorem~\ref{nonvanishing}, gives the existence of a sequence $(x_n)_n\subset\RN$
and constants $R, \zeta>0$ such that, up to a subsequence,
\begin{equation}\label{eqn:nonvanish_vn}
\int_{B_R(x_n)}|v_n|^{\crid}\dx 
\geq \|Q\|_\infty^{-\frac{\crid}{\crit}}\int_{B_R(x_n)}|Q^{\frac1{\crit}}v_n|^{\crid}\dx
\geq \|Q\|_\infty^{-\frac{\crid}{\crit}}\zeta>0,\quad\text{for all }n.
\end{equation}
Moreover, we may assume that $(x_n)_n\subset\Z^N$ by making $R$ larger if necessary. 
We now distinguish two cases.

{\bf Case 1}: $(x_n)_n$ is bounded. In this case, making $R$ again larger if necessary, 
we can assume that \eqref{eqn:nonvanish_vn} holds with $x_n=0$ for all $n$,
and we set $w_n:=v_n$ for all $n$ and $\wt{Q}:=Q$.

{\bf Case 2}: $|x_n|\to\infty$, for a subsequence. 
We restrict to this subsequence, setting $w_n:=v_n(\cdot+x_n)$ for all $n$ and $\wt{Q}:=Q_{\text{per}}$.

We want to apply Lemma \ref{lem:PS_Folgen} and therefore need to check that in both cases 
$ \left\|J_{\wt{Q}}'(w_{n})\mathds{1}_{B_{r}}\right\|_{\ast} \rightarrow 0 $ as $n \to \infty$, for all $r>0$. 
This is obvious in the first case, since $(v_{n})_{n}$ is a (PS)$_{\beta}$--sequence. 
For the second case, observe that for $r>0$ and $\varphi\in \ceinftyc(\RN)$, we have
\begin{equation}\label{eqn:Jprime_loc}
\begin{aligned}
J_{Q_{\text{per}}}'(w_n)\mathds{1}_{B_r}\varphi
&=J_{Q}'(v_n)\mathds{1}_{B_r(x_n)}\varphi(\cdot-x_n) 
+\int\limits_{\RN}\mathds{1}_{B_r}\varphi 
(\textbf{A}_{Q(\cdot+x_n)}-\textbf{A}_{Q_{\text{per}}(\cdot+x_n)})w_n\dx,
\end{aligned}
\end{equation}
using the fact that $Q_{\text{per}}$ is invariant under $\Z^N$-translations. 
Since $(v_n)_n$ is a Palais-Smale sequence, the first term goes to zero uniformly
for $\|\varphi\|_{\crid}\leq1$. The second term can be estimated as follows.
\begin{equation}\label{eqn:Jprime_quadratic}
\begin{aligned}
&\left|\ \int\limits_{\RN}\mathds{1}_{B_r}\varphi (\textbf{A}_{Q(\cdot+x_n)}-\textbf{A}_{Q_{\text{per}}(\cdot+x_n)})w_n\dx\right|\\
&\qquad=\left|\ \int\limits_{\RN}\mathds{1}_{B_r}\varphi (Q^{\frac1{\crit}}(\cdot+x_n)-Q_{\text{per}}^{\frac1{\crit}}(\cdot+x_n))
\textbf{R}\left[(Q^{\frac1{\crit}}(\cdot+x_n)+Q_{\text{per}}^{\frac1{\crit}}(\cdot+x_n))w_n\right]\dx\right|\\
&\qquad\leq 
2C_0 \|Q\|_\infty^{\frac1{\crit}} \left\|\varphi\right\|_{\crid} \|w_n\|_{\crid}
\left\|\mathds{1}_{B_r}(Q^{\frac1{\crit}}(\cdot+x_n)-Q_{\text{per}}^{\frac1{\crit}}(\cdot+x_n))\right\|_\infty,
\end{aligned}
\end{equation}
where $C_0>0$ is given by \eqref{eqn:c0}.
Moreover, by assumption, $Q(x)-Q_{\text{per}}(x)=Q_0(x)\to 0$ as $|x|\to\infty$. Thus,
since $|x_n|\to\infty$ as $n\to\infty$ and $Q, Q_{\text{per}}\geq 0$ are bounded functions, it follows that
\begin{equation}\label{eqn:asympt_QQp}
\|\mathds{1}_{B_r}(Q^{\frac1{\crit}}(\cdot+x_n)-Q_{\text{per}}^{\frac1{\crit}}(\cdot+x_n))\|_\infty\to 0, 
\quad\text{as }n\to\infty,\text{ for all }r>0.
\end{equation}
Combining \eqref{eqn:Jprime_loc}, \eqref{eqn:Jprime_quadratic} and \eqref{eqn:asympt_QQp}, we find that
$\|J_{Q_{\text{per}}}'(w_n)\mathds{1}_{B_r}\|_\ast\to 0$, as $n\to\infty$. 
Therefore the conditions of Lemma \ref{lem:PS_Folgen} are fulfilled in both cases
and, going to a subsequence, we obtain $w_n\weakto w$ in $L^{\crid}(\RN)$ and $w_n\to w$ a.e. on $\R^N$, 
for some $w\in L^{\crid}(\RN)$ which satisfies $J_{\wt{Q}}'(w)=0$.

Furthermore, from \eqref{eqn:BLnormintegral} we infer that, as $n\to\infty$,
\begin{equation}\label{eqn:c_loc}
\begin{aligned}
	\left\| \mathds{1}_{B_{r}}(w_{n}-w)\right\|^{\crid}_{\crid} 
	&= \int\limits_{\RN}\mathds{1}_{B_{r}}(w_{n} - w) \textbf{A}_{\wt{Q}}(w_{n} - w) \dx + o(1) \\
	&=\int\limits_{\RN}\mathds{1}_{B_{r}}(w_{n} - w) \textbf{A}_{\wt{Q}}\bigl[\mathds{1}_{B_{r}}(w_{n} - w)\bigr] \dx\\ 
	&\qquad\qquad+ \int\limits_{\RN}\mathds{1}_{B_{r}}(w_{n} - w) \textbf{A}_{\wt{Q}}\bigl[\mathds{1}_{M_{r}}(w_{n} - w)\bigr] \dx 
	+ o(1),
\end{aligned}
\end{equation}
for all $r>0$, where $M_r:=\RN\backslash B_r$.

For the first integral we obtain  with Proposition \ref{prop:lokkomp} and the characterization \eqref{eqn:MP_crit} of $L_Q^\ast$,
\begin{align*}
	&\int\limits_{\RN}  \mathds{1}_{B_{r}}(w_{n} - w) \textbf{A}_{\wt{Q}}\bigl[\mathds{1}_{B_{r}}(w_{n} - w)\bigr] \dx \\
	&\qquad=
	\int\limits_{\RN} \mathds{1}_{B_{r}}(w_{n} - w) \textbf{G}_{\wt{Q}}\bigl[\mathds{1}_{B_{r}}(w_{n} - w)\bigr] \dx 
	+ \intrn \mathds{1}_{B_{r}}(w_{n} - w) \left(\textbf{A}_{\wt{Q}} - \textbf{G}_{\wt{Q}}\right)
	\bigl[\mathds{1}_{B_{r}}(w_{n} - w)\bigr] \dx   \\
	&\qquad\leq
	\|Q\|_\infty^{\frac2{\crit}}\int\limits_{\RN} \mathds{1}_{B_{r}}|w_{n} - w|\textbf{R}_0\bigl[\mathds{1}_{B_{r}}|w_{n}-w|\bigr] \dx 
	+ o(1) \\
	&\qquad\leq 
	(N L_Q^\ast)^{-\frac2N}\left\|\mathds{1}_{B_{r}}(w_{n} - w)\right\|^{2}_{\crid} + o(1),\quad\text{as }n\to\infty,
\end{align*}
where $\mathbf{G}_{\wt{Q}}$ is given by \eqref{eqn:BS} with $\wt{Q}$ in place of $Q$. 
In addition, the Br\'ezis-Lieb Lemma implies
\begin{align*}
	\left\| \mathds{1}_{B_{r}}(w_{n} - w)\right\|^{\crid}_{\crid} &\leq \left\|w_{n} - w \right\|^{\crid}_{\crid} 
	 = \left\|w_{n} \right\|^{\crid}_{\crid} - \left\| w \right\|^{\crid}_{\crid} + o(1)  \\	
       &\leq \left\| v_{n} \right\|^{\crid}_{\crid} + o(1)
	 =\left( \frac{1}{\crid} - \frac{1}{2}\right)^{-1} \left(J_{Q}(v_{n}) - \frac{1}{2}J_{Q}'(v_{n})v_{n}\right) + o(1) \\
       &= N\beta + o(1),\quad\text{as }n\to\infty,
\end{align*}
since $(v_{n})_{n}$ is a $(PS)_{\beta}$-sequence by assumption.
Combining these two estimates, we obtain
\begin{align*}
	 \left[ 1 - \left( \frac{\beta}{L_Q^\ast}\right)^{\frac{2}{N}}\right]
	 &\left\| \mathds{1}_{B_{r}}(w_{n} - w)\right\|^{\crid}_{\crid} 
	 \leq
	  \left\| \mathds{1}_{B_{r}}(w_{n}-w)\right\|^{\crid}_{\crid} - 	(N L_Q^\ast)^{-\frac2N}
	  \left\|\mathds{1}_{B_{r}}(w_{n} - w)\right\|^{2}_{\crid} +o(1)\\
	 &\qquad\qquad\leq
	  \left\| \mathds{1}_{B_{r}}(w_{n}-w)\right\|^{\crid}_{\crid} 
	  - \int\limits_{\RN}\mathds{1}_{B_{r}}(w_{n} - w) \textbf{A}_{\wt{Q}}
	  \bigl[\mathds{1}_{B_{r}}(w_{n} - w)\bigr] \dx +o(1),
\end{align*}
as $n\to\infty$, and \eqref{eqn:c_loc} gives for all $r>0$,
\begin{equation}\label{eqn:estim_norm_HLS_asympt}
\left[ 1 - \left( \frac{\beta}{L_Q^\ast} \right)^{\frac{2}{N}}\right]
	\left\| \mathds{1}_{B_{r}}(w_{n} - w)\right\|^{\crid}_{\crid} 
\leq \int\limits_{\RN}\mathds{1}_{B_{r}}(w_{n} - w) \textbf{A}_{\wt{Q}}
	\bigl[\mathds{1}_{M_{r}}(w_{n} - w)\bigr] \dx + o(1),
\quad\text{as }n\to\infty,
\end{equation}
where the factor on the left-hand side is strictly positive, 
since we are assuming $\beta<L_Q^\ast$.

Let us now suppose by contradiction that $\left(\mathds{1}_{B_{r}}w_{n}\right)_{n}$ does not converge strongly
to $\mathds{1}_{B_{r}}w$ in $L^{\crid}(\RN)$, for some fixed $r > 0$. 
Then, passing to a subsequence there exists $\varepsilon > 0$ such that 
$$  
\left[ 1 - \left(   \frac{\beta}{L_Q^\ast} \right)^{\frac{2}{N}}\right]
\left\| \mathds{1}_{B_{r}}(w_{n} - w)\right\|^{\crid}_{\crid} > \varepsilon 
\qquad \text{for all }n.
$$
Lemma \ref{lem:technisch} applied to the sequence $\left(\wt{Q}^{\frac{1}{\crit}}(w_{n} - w) \right)_{n}$ gives 
$\rho_\eps>0$ such that for all $\rho\geq\max\{\rho_\eps,r\}$, we have
\begin{align*}
\varepsilon &> 
	\liminf\limits_{n\to\infty}\left|\ \intrn \mathds{1}_{B_{\rho}}(w_{n} - w) \textbf{A}_{\wt{Q}} 
		\bigl[\mathds{1}_{M_{\rho}}(w_{n} - w)\bigr] \dx\right| \\
	&\geq
	\liminf\limits_{n\to\infty} \left[ 1 - \left(   \frac{\beta}{L_Q^\ast} \right)^{\frac{2}{N}}\right]
		\left\| \mathds{1}_{B_{\rho}}(w_{n} - w)\right\|^{\crid}_{\crid},
	\qquad\text{using }\eqref{eqn:estim_norm_HLS_asympt} \\
	&\geq
	\liminf\limits_{n\to\infty} \left[ 1 - \left(  \frac{\beta}{L_Q^\ast} \right)^{\frac{2}{N}}\right]
		\left\| \mathds{1}_{B_{r}}(w_{n} - w)\right\|^{\crid}_{\crid} \\
	&\geq \varepsilon.
\end{align*}
This contradiction proves the strong convergence $\mathds{1}_{B_{r}}w_{n} \to \mathds{1}_{B_{r}}w$ in $L^{\crid}(\RN)$ 
as $n \to \infty$, for all $r>0$.
Using \eqref{eqn:nonvanish_vn} we immediately deduce $w \neq 0$. Moreover, 
\begin{equation}\label{eqn:MP_tilde}
\begin{aligned}
	J_{\wt{Q}}(w)=J_{\wt{Q}}(w)-\frac12J_{\wt{Q}}'(w)w=\frac1N\|w\|_{\crid}^\crid 
	&\leq \liminf_{n\to\infty}\frac1N\|w_n\|_{\crid}^\crid 
	= \liminf_{n\to\infty}\frac1N\|v_n\|_{\crid}^\crid\\ 
	&= \liminf_{n\to\infty}\bigl[J_{Q}(v_n)-\frac12J_{Q}'(v_n)v_n\bigr]=\beta=L_Q.
\end{aligned}
\end{equation}
In Case 1, the characterization~\eqref{eqn:MP_equivalent} of the mountain-pass level yields $J_Q(w)=L_Q$, 
and the Proposition is proven.
	
In Case 2, we consider the function 
$$
\wt{w}:=\left(\frac{Q_{\text{per}}}{Q}\right)^{\frac1{\crit}}w.
$$ 
Since $Q=Q_{\text{per}}+Q_0$ with $Q_0\geq 0$, we find that $|\wt{w}|\leq |w|$. 
In particular, we have $\wt{w}\in L^{2^+}(\R^N)$ and by definition,
$$
\int\limits_{\RN}\wt{w}\textbf{A}_Q\wt{w}\dx = \int\limits_{\RN}w\textbf{A}_{Q_{\text{per}}}w\dx=\|w\|_{\crid}^{\crid}>0,
$$
since $w$ is a nontrivial critical point of $J_{\wt{Q}}=J_{Q_{\text{per}}}$. Hence, $\wt{w}\neq 0$ and, setting
$$
\tau:=\left(\frac{\int\limits_{\RN}|\wt{w}|^{\crid}\dx}{\int\limits_{\RN}\wt{w}\textbf{A}_Q\wt{w}\dx}\right)^{\frac1{2-\crid}},
$$
we find that $0<\tau\leq 1$ and $J_Q'(\tau\wt{w})\wt{w}=0$. In addition, since $|\wt{w}|\leq |w|$ we have
\begin{align*}
J_Q(\tau\wt{w})=\frac1N\left(\frac{\|\wt{w}\|_{\crid}^2}{\int\limits_{\RN}\wt{w}\textbf{A}_Q\wt{w}\dx}\right)^{\frac{N}2}
\leq \frac1N\left(\frac{\|w\|_{\crid}^2}{\int\limits_{\RN}w\textbf{A}_{Q_{\text{per}}}w\dx}\right)^{\frac{N}2}
= J_{Q_{\text{per}}}(w).
\end{align*}
Therefore, \eqref{eqn:MP_equivalent} and \eqref{eqn:MP_tilde} yield $J_Q(\tau\wt{w})=L_Q=J_{Q_{\text{per}}}(w)$ 
and we deduce that $\tau=1$. 
We now claim that $\tau\wt{w}=\wt{w}$ is a critical point for $J_Q$.
To prove this, let $\varphi\in L^{\crid}(\RN)$ be arbitrarily given and choose $\delta>0$ such that 
$$
\int\limits_{\RN}(\wt{w}+s\varphi)\textbf{A}_Q(\wt{w}+s\varphi)\dx>0, \quad\text{for all }s\in[-\delta,\delta], 
$$
and, for $s\in[-\delta,\delta]$, set
$$
t_s:=\left(\frac{\int\limits_{\RN}|\wt{w}+s\varphi|^{\crid}\dx}{\int\limits_{\RN}(\wt{w}+s\varphi)\textbf{A}_Q(\wt{w}+s\varphi)\dx}\right)^{\frac1{2-\crid}}.
$$
Then we can write, using \eqref{eqn:MP_equivalent}, the property $J_Q(\wt{w})=J_Q(\tau\wt{w})=\max\limits_{t>0}J_Q(t\wt{w})$ 
and the mean-value theorem,
\begin{align*}
0\leq J_Q\bigl(t_s(\wt{w}+s\varphi)\bigr) - J_Q(\wt{w}) \leq J_Q\bigl(t_s(\wt{w}+s\varphi)\bigr) - J_Q(t_s\wt{w})
	=J_Q'\bigl(t_s(\wt{w}+s\sigma\varphi)\bigr)t_s s\varphi,
\end{align*}
for some $\sigma\in[-1,1]$. Dividing by $s\neq 0$ and letting $s\to 0^\pm$, we obtain $J_Q'(\wt{w})\varphi=0$, since $t_s\to 1$, as $s\to 0$.
This concludes the proof.
\end{proof}

\begin{remark}
(i) When $Q\equiv Q_0$, Case 2 does not occur in the proof above. Moreover, the operator $\mathbf{A}_Q-\mathbf{G}_Q$ is
itself compact, so that all arguments in the proof hold globally on $\RN$. As a consequence,
$J_Q$ satisfies the Palais-Smale condition at every level $0<\beta<L_Q^\ast$.

(ii) In the case where $Q\equiv Q_{\text{per}}$, we have $\wt{w}=w$ in the above proof and the Proposition is valid for any 
$0<\beta<L_Q^\ast$, except for the last assertion which should be replaced with $L_Q\leq J_Q(w)\leq\beta$.
\end{remark}

\subsection{Estimating the dual mountain-pass level}
Our next result shows that in dimension $N\geq 4$, the mountain pass level $L_Q$ lies below the critical 
threshold $L_Q^\ast$ if the coefficient $Q$ satisfies some flatness condition (see condition (Q) below). 
This additional condition seems to go back to the works of Escobar \cite{Escobar87} 
and Egnell \cite{Egnell88} (see also \cite[Remark 1.2]{Chabrowski-Szulkin2002}).

To estimate $L_Q$, we shall use the functions
\begin{equation}\label{eqn:dual_inst}
v_\eps(x):=(N(N-2)\eps)^{\frac{N+2}{4}}\left(\frac{1}{\eps + |x|^2}\right)^{\frac{N+2}{2}}, \quad \eps>0.
\end{equation}
It was shown by Lieb \cite{Lieb83} (see also \cite[Theorem 4.3]{Lieb_Loss2001}) that, up to translation and
multiplication by a constant, $v_\eps$, $\eps>0$ are the only optimizers of the Hardy-Littlewood-Sobolev 
inequality \eqref{eqn:HLS_inequ}, i.e.,
\begin{equation}\label{eqn:HLS_optimal}
\int\limits_{\RN}v_\eps \mathbf{R}_0 v_\eps \dx = S^{-1}\|v_\eps\|_{\crid}^2.
\end{equation}
In addition, we have $v_\eps=u_\eps^{\crit-1}$, where 
\begin{equation*}
u_{\eps}(x) := (N(N-2)\eps)^{\frac{N-2}{4}}\left(\frac{1}{\eps + |x|^2}\right)^{\frac{N-2}{2}}, \quad \eps>0,
\end{equation*}
are the Aubin-Talenti instantons (see, e.g., \cite{Brezis-Nirenberg83}) that optimize 
the Sobolev inequality \eqref{eqn:Sobolev_inequ} 
and for which the following holds:
$$
\left\| \nabla u_{\eps}\right\|^{2}_{L^{2}(\RN)} = \left\|u_{\eps}\right\|^{\crit}_{\crit} = S^\frac{N}{2}, \quad\text{for all }\eps>0.
$$
In particular, we deduce that
\begin{equation}\label{eqn:norm_dual_instanton}
\|v_\eps\|_{\crid}=\|u_\eps\|_{\crit}^{\crit-1}=S^{\frac{N+2}4}, \quad\text{for all }\eps>0.
\end{equation}
\begin{proposition}\label{prop:MP_level}
Let $N\geq 4$ and consider $Q \in L^{\infty}(\RN) \setminus\{0\}$ nonnegative. Assume further, 
that there exists $x_{0} \in \RN$ with $ Q(x_{0}) = \left\|Q\right\|_\infty$ and that
\begin{equation}
\label{Q}\tag{Q}
	Q(x_0) - Q(x) = o(|x - x_{0}|^2),\quad \text{ as } |x-x_0|\to 0.
\end{equation}
Then we have
$$
	L_{Q}< L_Q^\ast.
$$
\end{proposition}
\begin{proof}[Proof of Proposition~\ref{prop:MP_level} ]
Let us assume -- without loss of generality -- that $x_{0} = 0$
and set $q:=\|Q\|_\infty$. 

We consider for $\eps>0$ the dual instanton $v_\eps$ given by \eqref{eqn:dual_inst} and
put $v:=v_1$.
Fix a cut-off function $\varphi \in \ceinftyc(\RN)$ with $0 \leq \varphi \leq 1$ on $\RN$,  
$\varphi \equiv 1$ on $B_{1}(0)$ and $\varphi \equiv 0$ outside of $B_{2}(0)$. 
Setting for $\eps>0$, $\alpha>0$, 
$$
v_{\eps,\alpha}:=\varphi_\alpha v_{\eps},\quad\text{where }\varphi_\alpha(x):=\varphi\bigl(\frac{x}{\alpha}\bigr),
$$
we shall estimate the ratio
$$ 
\frac{\left\|v_{\eps,\alpha}\right\|^{2}_{\crid}}{\int\limits_{\RN}v_{\eps,\alpha}\textbf{A}_{Q}v_{\eps,\alpha}\dx}, 
$$
and we first look at the quadratic form $\int\limits_{\RN} v_{\eps,\alpha}\textbf{A}_{Q}v_{\eps,\alpha}\dx$. 
Consider the decomposition
\begin{equation}\label{eqn:decomp_bilinear}
\begin{aligned}
\int\limits_{\R^{N}}v_{\eps,\alpha}\textbf{A}_Qv_{\eps,\alpha}~dx 
&= \int\limits_{\R^{N}}v_{\eps}\textbf{G}_qv_{\eps}~dx 
-\int\limits_{\R^{N}}(1+\varphi_\alpha)v_{\eps}\textbf{G}_q\bigl((1-\varphi_\alpha)v_{\eps}\bigr)~dx\\
&\quad + \int\limits_{\R^{N}}v_{\eps,\alpha}\bigl(\textbf{A}_q-\textbf{G}_q\bigr)v_{\eps,\alpha}~dx 
- \int\limits_{\R^{N}}v_{\eps,\alpha}\bigl(\textbf{A}_q-\textbf{A}_Q\bigr)v_{\eps,\alpha}~dx, 
\end{aligned}
\end{equation}
with $\textbf{G}_q$ as in \eqref{eqn:BS} where $Q$ is replaced by the constant function $q$, i.e., 
$\textbf{G}_q=q^{\frac1{\crit}}\textbf{R}_0q^{\frac1{\crit}}$.
Starting with the first integral in the right-hand side of \eqref{eqn:decomp_bilinear}, we remark that \eqref{eqn:HLS_optimal}
and \eqref{eqn:norm_dual_instanton} together with the definition of $\mathbf{G}_q$ give
\begin{equation}\label{eqn:K2}
\begin{aligned}
\int\limits_{\RN}v_{\eps}\textbf{G}_qv_{\eps} \dx 
=q^{\frac2{\crit}}\int\limits_{\RN} v_\eps\mathbf{R}_0 v_\eps \dx
=q^{\frac2{\crit}} S^{\frac{N}{2}}.
\end{aligned}
\end{equation}
Using the Hardy-Littlewood-Sobolev inequality,
the second integral in \eqref{eqn:decomp_bilinear} can be estimated as follows
\begin{align*}
\int\limits_{\R^{N}}(1+\varphi_\alpha)v_{\eps}\textbf{G}_q\bigl((1-\varphi_\alpha)v_{\eps}\bigr) \dx
&\leq q^{\frac2{\crit}}S^{-1}\left\|(1+\varphi_\alpha)v_{\eps}\right\|_\crid 
\left\|(1-\varphi_\alpha)v_{\eps}\right\|_\crid.
\end{align*}
Moreover, since $1-\varphi_\alpha=0$ in $B_\alpha(0)$, we obtain
\begin{align*}
\left\|(1-\varphi_\alpha)v_{\eps}\right\|^{\crid}_{\crid} 
\leq N\omega_N (N(N-2))^{\frac{N}{2}} \int\limits_{\frac{\alpha}{\sqe}}^\infty r^{-(N+1)}~dr 
=\omega_N(N(N-2))^{\frac{N}{2}} \alpha^{-N}\eps^{\frac{N}2}.
\end{align*}
Thus, from \eqref{eqn:norm_dual_instanton} and since $0\leq\varphi_\alpha\leq 1$, it follows that
\begin{equation}\label{eqn:G2_Ieps}
\begin{aligned}
\int\limits_{\R^{N}}(1+\varphi_\alpha)v_{\eps}\textbf{G}_q\bigl((1-\varphi_\alpha)v_{\eps}\bigr)\dx 
\leq 2q^{\frac2{\crit}}(\omega_N)^{\frac1{\crid}} S^{\frac{N-2}{4}}(N(N-2))^{\frac{N+2}{4}} \alpha^{-\frac{N+2}{2}}
\eps^{\frac{N+2}4}.
\end{aligned}
\end{equation}
The third integral in \eqref{eqn:decomp_bilinear} can be rewritten as
\begin{align*}
\int\limits_{\RN}v_{\eps,\alpha}\bigl(\textbf{A}_q-\textbf{G}_q\bigr)v_{\eps,\alpha}\dx
&=q^{\frac{2}{\crit}}\int\limits_{\RN}v_{\eps,\alpha}\bigl[(\Psi-\Lambda)\ast v_{\eps,\alpha}\bigr]\dx\\
&=q^{\frac{2}{\crit}}\int\limits_{\RN}\int\limits_{\RN}v_{\eps}(x)v_{\eps}(y)\varphi_\alpha(x)\varphi_\alpha(y)
\bigl[\Psi(x-y)-\Lambda(x-y)\bigr] \dy\dx.
\end{align*}
Since $\varphi_\alpha(x)=0$ for all $|x|\geq 2\alpha$, it is enough to estimate the difference $\Psi-\Lambda$
inside $B_{4\alpha}(0)$. Fixing $\alpha_0\in(0,\frac14 y_0)$ and observing that $y_{\nu}<y_{\nu+1}$ for $\nu\geq 0$, 
we obtain from Lemma~\ref{lem:difference_fundsol} a constant $\kappa_0>0$ such that
\begin{align*}
\Psi(z)-\Lambda(z) \geq \left\{\begin{array}{ll} \kappa_0 |z|^{4-N},& \text{if }N\geq 5,\\ \\
								         \kappa_0 \bigl|\ln|z|\, \bigr|, & \text{if }N=4,
						\end{array}\right.
\quad\text{for all }0<|z|\leq 4\alpha_0.
\end{align*}
As a consequence, and since $\varphi_\alpha\equiv 1$ in $B_\alpha$, we can write for all $0<\alpha\leq \alpha_0$
and $0<\eps\leq\alpha^2$:
\begin{align*}
\int\limits_{\RN}v_{\eps,\alpha}\bigl(\textbf{A}_q-\textbf{G}_q\bigr)v_{\eps,\alpha}\dx
&\geq \kappa_0 q^{\frac{2}{\crit}}\int\limits_{B_\alpha}\int\limits_{B_\alpha}v_{\eps}(x)v_{\eps}(y)
|x-y|^{4-N} \dy\dx\\
&=\eps\kappa_0 q^{\frac{2}{\crit}}\int\limits_{B_{\frac{\alpha}{\sqrt{\eps}}}}\int\limits_{B_{\frac{\alpha}{\sqrt{\eps}}}}v(x)v(y)
|x-y|^{4-N} \dy\dx\\
&\geq \eps 2^{4-N}\kappa_0 q^{\frac{2}{\crit}}\left(\int\limits_{B_1}v(x)\, dx\right)^2,
\end{align*}
in the case where $N\geq 5$. In a similar way, we obtain for $N=4$,
\begin{equation}\label{eqn:diff_N4}
\int\limits_{\RN}v_{\eps,\alpha}\bigl(\textbf{A}_q-\textbf{G}_q\bigr)v_{\eps,\alpha}\dx
\geq \eps|\ln(2\sqrt{\eps})|\kappa_0 q^{\frac{2}{\crit}}\left(\int\limits_{B_1}v(x)\, dx\right)^2.
\end{equation}
Setting $\gamma:= 2^{4-N}\kappa_0 q^{\frac{2}{\crit}}\left(\int\limits_{B_1}v(x)\, dx\right)^2$, the above computations
yield
\begin{equation}\label{eqn:H1}
\int\limits_{\RN}v_{\eps,\alpha}\bigl(\textbf{A}_q-\textbf{G}_q\bigr)v_{\eps,\alpha}\dx
\geq \gamma \eps, 
\qquad\text{for all }0<\alpha\leq\alpha_0\text{ and }0<\eps\leq\min\bigl\{\alpha^2,\frac{e^{-2}}4\bigr\}.
\end{equation}
To estimate the remaining integral in \eqref{eqn:decomp_bilinear}, we first note that, since $0\leq \frac{Q(x)}{q}\leq 1$, we have
$$
0\leq q^{\frac{1}{\crit}}-Q^{\frac{1}{\crit}}(x)\leq q^{\frac1{\crit}-1}\bigl(q-Q(x)\bigr),\quad\text{ for all }x.
$$
Thus, the assumption \eqref{Q} gives
for each $\delta>0$ a constant $\alpha_\delta>0$ such that
$$
0\leq q^{\frac{1}{\crit}}-Q^{\frac{1}{\crit}}(x)\leq \frac{\delta}{2} |x|^2, \quad\text{for all }|x|\leq 2\alpha_\delta.
$$
Since $\varphi_{\eps,\alpha}\equiv 0$ outside $B_{2\alpha}$, we find for $0<\alpha\leq\alpha_\delta$ and $0<\eps\leq\alpha^2$,
\begin{align*}
\int\limits_{\RN}v_{\eps,\alpha}\bigl(\textbf{A}_q-\textbf{A}_Q\bigr)v_{\eps,\alpha}\dx
&= \int\limits_{\RN}(q^{\frac{1}{\crit}}-Q^{\frac{1}{\crit}})v_{\eps,\alpha}\textbf{R}\bigl[(q^{\frac{1}{\crit}}+Q^{\frac{1}{\crit}})v_{\eps,\alpha}\bigr]\dx\\
&\leq 2q^{\frac{1}{\crit}} C_0 \bigl\| (q^{\frac{1}{\crit}}-Q^{\frac{1}{\crit}})v_{\eps,\alpha} \bigr\|_{\crid} \|v_{\eps,\alpha}\|_{\crid}\\
&\leq \delta\eps q^{\frac{1}{\crit}} S^{\frac{N+2}{4}} C_0\left(\int_{\RN}\Bigl|\, |x|^2 v(x)\Bigr|^{\crid}\, dx\right)^{\frac1{\crid}}.
\end{align*}
Choosing $\delta>0$ such that 
$\delta q^{\frac{1}{\crit}} S^{\frac{N+2}{4}} 
C_0\left(\int_{\RN}\Bigl|\, |x|^2 v(x)\Bigr|^{\crid}\, dx\right)^{\frac1{\crid}}\leq \frac{\gamma}{2}$
and setting $\alpha:=\min\{\alpha_0,\alpha_\delta\}$, we obtain the estimate
\begin{equation}\label{eqn:H2}
\int\limits_{\RN}v_{\eps,\alpha}\bigl(\textbf{A}_q-\textbf{A}_Q\bigr)v_{\eps,\alpha}\dx \leq \frac{\gamma}{2}\eps,
\quad\text{ for all }0<\eps\leq\alpha^2.
\end{equation}
With this choice of $\alpha$, putting the estimates \eqref{eqn:K2}, \eqref{eqn:G2_Ieps}, \eqref{eqn:H1} and \eqref{eqn:H2}
together, the decomposition \eqref{eqn:decomp_bilinear} yields
\begin{equation}\label{eqn:bilin_abschaetzung}
\begin{aligned}
\int\limits_{\RN}v_{\eps,\alpha}\textbf{A}_Qv_{\eps,\alpha}\dx 
&\geq q^{\frac2{\crit}} S^{\frac{N}2} + \frac{\gamma}{2}\eps -\zeta\eps^{\frac{N+2}4}
\geq q^{\frac2{\crit}} S^{\frac{N}2} + \frac{\gamma}{4}\eps\\
&>q^{\frac2{\crit}} S^{\frac{N}2},
\quad\text{for }0<\eps\leq\eps_0:=\min\bigl\{\alpha^2,\bigl(\frac{\gamma}{4\zeta}\bigr)^{\frac{4}{N-2}}, \frac{e^{-2}}{4}\bigr\}, 
\end{aligned}
\end{equation}
where $\zeta=2q^{\frac2{\crit}} (\omega_N)^{\frac1{\crit}} S^{\frac{N-2}{4}}(N(N-2))^{\frac{N+2}{4}} \alpha^{-\frac{N+2}{2}}$.
Hence, from \eqref{eqn:MP_equivalent}, \eqref{eqn:MP_crit_sobolev}, \eqref{eqn:norm_dual_instanton} and 
\eqref{eqn:bilin_abschaetzung} we infer that 
for $\alpha=\min\{\alpha_0,\alpha_\delta\}$ and $0<\eps\leq\eps_0$,
$$
L_Q
\leq \frac1N\left( \frac{\left\|v_{\eps,\alpha}\right\|^{2}_{\crid}}{
\int\limits_{\RN}v_{\eps,\alpha}\textbf{A}_{Q}v_{\eps,\alpha}\dx}\right)^{\frac{N}{2}}
<\frac1N\left(\frac{S^{\frac{N+2}{2}}}{q^{\frac{2}{\crit}}S^{\frac{N}{2}}}\right)^{\frac{N}{2}}
=\frac{S^{\frac{N}{2}}}{N q^{\frac{N-2}{2}}}=L_Q^\ast.
$$
This proves the desired result.
\end{proof}

\begin{remark}\label{remq:N4}
In the case $N=4$, using the estimate \eqref{eqn:diff_N4} instead of \eqref{eqn:H1},
we see that the condition \eqref{Q} can be weakened to
$$
Q(x_0) - Q(x) = O(|x - x_{0}|^2),\quad \text{ as } |x-x_0|\to 0.
$$
\end{remark}

\subsection{Existence and nonexistence of dual ground states}
We are now in position to give the proof of our main existence result for the critical nonlinear Helmholtz equation.
\begin{theorem}
Let $N\geq 4$ and consider $Q\in L^\infty(\R^N)\backslash\{0\}$. 
Assume in addition that
\begin{itemize}
\item[(Q1)] $Q=Q_{\text{per}}+Q_0$, where $Q_{\text{per}}, Q_0\geq 0$ are such that $Q_{\text{per}}$ is $\Z^N$-periodic 
and $Q_0(x)\to 0$ as $|x|\to\infty$;
\item[(Q2)] there exists $x_{0} \in \RN$ with $Q(x_{0}) = \left\|Q\right\|_\infty$ and, as $|x-x_0|\to 0$,
$$
Q(x_0) - Q(x) = \begin{cases} o(|x - x_{0}|^2),\quad \text{if }N\geq 5, \\ \\ O(|x-x_0|^2), \quad\text{if }N=4.\end{cases} 
$$
\end{itemize}
Then the problem
\begin{equation}\label{eqn:nlh_critical}
-\Delta u -u=Q(x)|u|^{\crit-2}u, \quad u \in W^{2,\crit}(\RN)
\end{equation}
has a dual ground state.
\end{theorem}
\begin{proof} 
Using the mountain pass theorem without the Palais-Smale condition (see \cite{Ambrosetti-Rabinowitz73} and
 \cite[Theorem 2.2]{Brezis-Nirenberg83}) we obtain the existence of a Palais-Smale sequence 
$(v_{n})_{n} \subset L^{\crid}(\RN)$ at the mountain pass level $L_{Q}$.
Therefore, by Proposition~\ref{prop:PS_unter_kritisch_Level}, Proposition~\ref{prop:MP_level} and Remark~\ref{remq:N4}, 
the functional $J_Q$ possesses a critical point $w \in L^{\crid}(\RN)$ of $J_{Q}$ which satisfies $J_Q(w)=L_Q$.
Setting $u = \mathbf{R}\bigl(Q^{\frac{1}{\crit}}w\bigr)$, we find that $u \in L^{\crit}(\RN)$ is a dual ground state 
of \eqref{eqn:nlh_critical}, and this concludes the proof.
\end{proof}

In dimension $N=3$, the situation completely changes. Indeed, the proof of Proposition~\ref{prop:MP_level} fails, since
the estimate in Lemma~\ref{lem:difference_fundsol}(i) now has the opposite sign. In fact, we have the following nonexistence result.
\begin{proposition}\label{prop:nonexist_N3}
Let $Q\in L^\infty(\R^3)\backslash\{0\}$ satisfy $Q(x)\geq 0$ for almost every $x\in\R^3$.
Then, there is no dual ground state for the problem
\begin{equation}\label{eqn:nlh_critical3}
-\Delta u -u=Q(x)|u|^{\crit-2}u, \quad u \in W^{2,\crit}(\R^3).
\end{equation}
\end{proposition}
\begin{proof}
%
Consider the family of functions $v_{\eps}$, $\eps>0$, given in \eqref{eqn:dual_inst}.
Since $|\Psi(x)-\Lambda(x)|\to 0$ as $|x|\to 0$, we can find find for every $0<\delta<\|Q\|_\infty$, a point $x_0\in\R^N$ 
and cutoff function $\varphi\in \ceinftyc(\RN)$ such that $\varphi(x_0)=1$ and
$$
\liminf_{\eps\to 0^+}\int_{\R^3}\varphi v_{\eps}(\cdot-x_0)\mathbf{A}_Q\left(\varphi v_{\eps}(\cdot-x_0)\right)\dx \geq
 (\|Q\|_\infty-\delta)^{\frac2{\crit}} \int_{\R^3}v_1 \mathbf{R}_0 v_1\, dx = (\|Q\|_\infty-\delta)^{\frac2{\crit}} S^{-1} \|v_1\|_{\crid}^2.
$$
In addition, since $\|\varphi v_\eps(\cdot-x_0)\|_{\crid}\to \|v_1\|_{\crid}$, as $\eps\to 0^+$, the characterization \eqref{eqn:MP_equivalent} 
of $L_Q$ gives
$$
L_Q\leq \limsup_{\eps\to 0^+}\frac1N\left(\frac{\|\varphi v_\eps(\cdot-x_0)\|_{\crid}^2}{
\int_{\R^3}\varphi v_{\eps}(\cdot-x_0)\mathbf{A}_Q\left(\varphi v_{\eps}(\cdot-x_0)\right)\dx}\right)^{\frac{N}2}
\leq \frac{S^{\frac{N}2}}{N(\|Q\|_\infty-\delta)^{\frac{N+2}2}}.
$$
Since $\delta>0$ can be chosen arbitrarily small, we infer from \eqref{eqn:MP_crit_sobolev} that $L_Q\leq L_Q^\ast$.

Let us now assume, by contradiction, that $L_Q$ is achieved. In this case, there exists $v\in L^{\crid}(\R^3)$ such that
$\|v\|_\crid=1$ and
$$
\int_{\R^3}v\mathbf{A}_Qv\dx=(NL_Q)^{-\frac2N}.
$$
Since $L_Q\leq L_Q^\ast$ and recalling the value of $L_Q^\ast$ given in \eqref{eqn:MP_crit_sobolev}, we can write
\begin{align*}
S^{-1}\|Q\|_\infty^{\frac2\crit}
&=(NL_Q^\ast)^{-\frac2N}\leq (NL_Q)^{-\frac2N}=\int_{\R^3}v\mathbf{A}_Qv\dx\\ 
&\leq \int_{\R^3}Q^{\frac1\crit}|v| \left[|\Psi|\ast(Q^{\frac1\crit}|v|)\right]\dx
\leq \int_{\R^3}Q^{\frac1\crit}|v| \left[\Lambda\ast(Q^{\frac1\crit}|v|)\right]\dx\\
&= \int_{\R^3}Q^{\frac1\crit}|v| \mathbf{R}_0\left(Q^{\frac1\crit}|v|\right)\dx
\leq S^{-1}\|Q^{\frac1\crit}v\|_\crid^2
\leq S^{-1}\|Q\|_\infty^{\frac2\crit},
\end{align*}
using the fact that $|\Psi(z)|=\frac{|\cos|z|\, |}{4\pi|z|}\leq \frac1{4\pi|z|}=\Lambda(z)$ for all $z\in\R^N$, and
the Hardy-Littlewood-Sobolev inequality.
As a consequence, all inequalities are equalities and we find $L_Q=L_Q^\ast$ and obtain the following identities.
\begin{align}
 \int_{\R^3}Q^{\frac1\crit}|v| \mathbf{R}_0\left(Q^{\frac1\crit}|v|\right)\dx
 &= S^{-1}\|Q^{\frac1\crit}v\|_\crid^2, \label{eqn:equal_HLS}\\
 \int_{\R^3}Q^{\frac1\crit}|v| \left[|\Psi|\ast(Q^{\frac1\crit}|v|)\right]\dx
&= \int_{\R^3}Q^{\frac1\crit}|v| \left[\Lambda\ast(Q^{\frac1\crit}|v|)\right]\dx.\label{eqn:equal_Helm_Lapl}
\end{align}
From \eqref{eqn:equal_HLS} and the uniqueness of the optimizers for the Hardy-Littlewood-Sobolev inequality \cite{Lieb83,Lieb_Loss2001},
we deduce that
$$
Q^{\frac1\crit}|v|=\gamma v_\eps(\cdot-x_0), \quad\text{for some }\gamma, \eps>0\text{ and }x_0\in\R^N,
$$
where $v_\eps$ is given by \eqref{eqn:dual_inst}.
In particular, $Q^{\frac1\crit}|v|>0$ everywhere in $\R^3$, and we obtain
\begin{align*}
\int_{\R^3}Q^{\frac1\crit}|v| \left[|\Psi|\ast(Q^{\frac1\crit}|v|)\right]\dx
&=\int_{\R^3}\int_{\R^3}Q^{\frac1\crit}(x)|v(x)|Q^{\frac1\crit}(y)|v(y)|\frac{\bigl|\cos|x-y|\, \bigr|}{4\pi|x-y|}\dx\dy\\
&<\int_{\R^3}\int_{\R^3}Q^{\frac1\crit}(x)|v(x)|Q^{\frac1\crit}(y)|v(y)|\frac{1}{4\pi|x-y|}\dx\dy\\
&=\int_{\R^3}Q^{\frac1\crit}|v| \left[\Lambda\ast(Q^{\frac1\crit}|v|)\right]\dx.
\end{align*}
This contradicts \eqref{eqn:equal_Helm_Lapl} and therefore shows that $L_Q$ is not achieved.
In particular, $J_Q$ does not have any critical point at level $L_Q$, and thus no dual ground state solution 
of \eqref{eqn:nlh_critical3} can exist.
\end{proof}

\section*{Acknowledgments}
The authors would like to thank Tobias Weth for suggesting the problem studied here and for stimulating discussions.
This research was partially supported by the {\em Deutsche Forschungsgemeinschaft (DFG)} through the project WE 2821/5-2.

\bibliographystyle{siam}
\bibliography{literatur1}

\end{document}